\documentclass[a4paper, 11pt,reqno]{amsart}

\usepackage{amsmath,amsthm,amsbsy,amssymb}
\usepackage{arydshln} 
\usepackage{ulem} 
\usepackage{booktabs} 

\makeatletter
\newcommand{\leqnos}{\tagsleft@true\let\veqno\@@leqno}
\newcommand{\reqnos}{\tagsleft@false\let\veqno\@@eqno}
\reqnos
\makeatother

\usepackage{xcolor} 
\definecolor{orange}{rgb}{1,0.5,0}
\definecolor{Red}{rgb}{.795,0.015,0.017}
\definecolor{Ggreen}{rgb}{0.,0.675,0.0128}
\definecolor{Bblue}{rgb}{0.16,.32,0.91}
 \usepackage[backref=page]{hyperref}
\hypersetup{
    colorlinks = true,
linkcolor={red},
urlcolor={blue},
citecolor={Ggreen},    
urlcolor = {blue},
citebordercolor = {0.33 .58 0.33},
 linkbordercolor = {0.99 .28 0.23},
 pdftitle={SUPER-REGULAR POLYTOPES IN CYCLOTOMIC HYPERCUBES},
 pdfauthor={CZ}
}
\usepackage{mathabx}

\usepackage{mathrsfs}

\newcommand{\scrB}{\mathscr B}

\newcommand{\cI}{\mathcal I}
\newcommand{\cU}{\mathcal U}
\newcommand{\cV}{\mathcal V}
\newcommand{\cW}{\mathcal W}
\newcommand{\cT}{\mathcal T}

\usepackage{bm}
\newcommand*{\B}[1]{\ifmmode\bm{#1}\else\textbf{#1}\fi}

\newcommand{\bx}{\B{x}}
\newcommand{\bbeta}{\B{\beta}}
\newcommand{\balpha}{\B{\alpha}}

\newcommand{\bc}{\B{c}}

\def\CC{\mathbb{C}}

\def\RR{\mathbb{R}}
\def\NN{\mathbb{N}}
\def\ZZ{\mathbb{Z}}
\def\QQ{\mathbb{Q}}

\newcommand{\sdfrac}[2]{\mbox{\small$\displaystyle\frac{#1}{#2}$}}

\usepackage{xspace}

\DeclareMathOperator*{\NP}{NP}
\DeclareMathOperator*{\SP}{SP}
\DeclareMathOperator*{\EP}{EP}
\DeclareMathOperator*{\WP}{WP}

\DeclareMathOperator{\distance}{\mathfrak{d}} 
\newcommand{\fd}{\distance}

\DeclareMathOperator*{\diam}{diam}
\DeclareMathOperator*{\Gal}{Gal}

\usepackage{float}
\newcommand{\Tr}{\mathop{\mathrm{Tr}}\nolimits}

\newcommand{\abs}[1]{\left\vert #1 \right\vert}

\newcommand{\nnorm}[1]{\left|\hspace*{2.5pt} \!\!\left| #1\right|\hspace*{2.5pt} \!\!\right|}

\theoremstyle{plain}
\newtheorem{theorem}{Theorem}
\newtheorem{corollary}{Corollary}

\newtheorem{lemma}{Lemma}[section]

\theoremstyle{remark}
\newtheorem{remark}{Remark}[section]
\newtheorem*{remark*}{Remark}

\theoremstyle{definition}

\usepackage{subfigure}  
\usepackage{float}
\usepackage[pdftex]{graphicx}
\graphicspath{{IMAGES/}{./}}
\usepackage[]{caption}
\renewcommand*{\backref}[1]{}
\renewcommand*{\backrefalt}[4]{%
  \ifcase #1 %
No citations.
  \or
(page #2).%
  \else
(pages #2).%
  \fi%
}
\usepackage{cite}

\begin{document}

\title[Super-regular polytopes in cyclotomic hypercubes]
{Super-regular polytopes in cyclotomic hypercubes}

\author[C. Cobeli, A. Zaharescu]
{Cristian Cobeli, Alexandru Zaharescu}

\address[CC, AZ]{\normalfont 
``Simion Stoilow'' Institute of Mathematics of the Romanian Academy,~21 Calea Griviței Street, 
P. O. Box 1-764, Bucharest 014700, Romania}
\address[AZ]{\normalfont
Department of Mathematics,
University of Illinois at Urbana-Champaign,
Altgeld Hall, 1409 W. Green Street,
Urbana, IL, 61801, USA\vspace{7pt}}

\email{cristian.cobeli@imar.ro}  
\email{zaharesc@illinois.edu}  

\makeatletter
\@namedef{subjclassname@2020}{%
  \textup{2020} Mathematics Subject Classification}
\makeatother
\subjclass[2020]{Primary 11B99; Secondary 11P21, 52B11}

\keywords{Cyclotomic hypercubes, super-regular polytopes, Euclidean distance}

\begin{abstract}
For any odd prime $p$ and any integer $N\ge 0$, let $\cV(p,N)$ be the set of vertices
of the cyclotomic box $\scrB = \scrB(p,N)$ of edge size $2N$ and centered at the origin $O$ of the ring of integers $\ZZ[\omega]$
of the cyclotomic field $\QQ(\omega)$, 
where $\omega=\exp\big(\frac{2\pi i}{p}\big)$.
Cyclotomic boxes represented as sets of points in the complex plane prove to have counter-intuitive super-regularity properties that are known to occur in high dimensional real hypercubes.

Employing the naturally induced Euclidean-trace metric for distance  measurement and letting the prime $p$ tend to infinity, we prove the following results.
1. Almost all triangles with vertices in $\cV(p,N)$ are almost equilateral.
2. Almost all angles $\angle VOA$, where $V$ is in~$\cV(p,N)$, $O$ is the origin, which coincides with the center of $\scrB(p,N)$, and $A$ is fixed anywhere  in $\scrB(p,N)$, are right angles.
3. Almost all pyramids with base on $\cV(p,N)$ and the apex fixed anywhere in $\scrB(p,N)$
are super-regular, meaning that the base has all edges and diagonals almost equal and 
the lateral faces are nearly isosceles triangles, each nearly equal to the others.
\end{abstract}
\maketitle

\section{Introduction}
The study of discrete versions of classical problems in geometric measure 
theory, harmonic analysis, number theory, or probability 
has been a central point of research in recent decades.
For a brief selection of papers in the area of interest of this work, 
where the problems are investigated across different dimensions and 
underlying structures
see~\cite{OR2019, BP2009, ES1996, ISX2010, MMP1999, KPP2023, Shp2016}.
The discrete nature of these problems have found diverse 
practical applications 
in optimizing communication networks~\cite{Hae2013, SH2010},
data mining~\cite{AHK2001},
the design of neural networks~\cite{BS2021}, 
or statistical mechanics and string theory~\cite{Sel2021,ARR2018}.

Within multidimensional spaces, common sense-trained intuition can be tricked when faced with the truth. 
Instances where this can happen include phenomena similar to the one used
by Bubeck and Sellke~\cite{BS2021} to argue the necessity 
for neural networks to be larger than previously expected in order to 
avoid fundamental mishaps.
Their proof is based on the observation that almost all pairs of points 
placed on a high-dimensional sphere are almost a diameter apart 
from each other. 
Related phenomena are also known to occur in high-dimensional comparable-size normalized hypercubes.
There, despite the huge number of triangular shapes, which grows fast 
as the dimension increases, it is found that almost all triangles are 
almost equilateral~\cite{ACZ2024,ACZ2023}.

In this paper, we show that this phenomenon 
can be encountered concealed even in two dimensions, 
within a complex plane setting.
Let $p$ be an odd prime and let $\QQ(\omega)$
denote the cyclotomic field generated by the 
primitive $p$-th root of unity $\omega=\exp\big(\frac{2\pi i}{p}\big)$. 
For a positive integer~$N$, we denote by $\scrB(p,N)$ the 
symmetric \textit{box of cyclotomic lattice points}
\begin{equation*}
    \scrB(p,N):=\big\{a_1\omega+\cdots+a_{p-1}\omega^{p-1}
    : a_1,\dots,a_{p-1}\in [-N,N]\cap\ZZ
    \big\}.
\end{equation*}
The vertices of $\scrB(p,N)$ are in the set
\begin{equation*}
    \cV(p,N):=\big\{a_1\omega+\cdots+a_{p-1}\omega^{p-1}
    : a_1,\dots,a_{p-1}\in \{-N,N\}
    \big\}.
\end{equation*}
Notice that the size of  these sets increase exponentially with $p$ 
since \mbox{$\#\scrB(p,N)=(2N+1)^{p-1}$} and $\#\cV(p,N)=2^{p-1}$.
An indication of how these \textit{cyclotomic hypercubes} look like can be seen 
in Figures~\ref{FigureRandomPolygons} and~\ref{FigurePoles13-16}.
For any $\gamma\in\QQ(\omega)$, 
after its successive rotation through the elements of the 
$\omega$-basis $\omega,\dots,\omega^{p-1}$, we let
\begin{align*}
  \psi(\gamma):=\big(\Tr(\gamma\omega),\dots,
\Tr(\gamma\omega^{p-1})\big)
\end{align*}
be the\textit{ trace-map image} of~$\gamma$ 
in the vector space $\QQ^{p-1}$.
Then, let $d(\alpha,\beta)$ be the \textit{cyclotomic distance} between 
 $\alpha$ and $\beta$, which is defined as the 
the Euclidean distance between their $\psi$-embeddings, that is, 
\begin{align}\label{eqDistanceTrace}
  d(\alpha,\beta):=\nnorm{\psi(\beta)-\psi(\alpha)}_E.  
\end{align} 
This cyclotomic distance, referred to henceforth briefly as distance, is a suitable distance in~$\QQ(\omega)$, 
which finely complies with the algebraic-arithmetic structure.
Then the associated norm of $\alpha\in\QQ(\omega)$, which is the distance from $\alpha$ to the origin $O$, 
can be nicely expressed (see Lemma~\ref{LemmaNormFormula})
in terms of the trace of $\alpha$
and the Euclidean norm of $\balpha\in\QQ^{p-1}$,
the vector determined by
the coefficients of $\alpha$ in the $\omega$-basis.
What is noteworthy is that while the properties of the norm are preserved for any $p$, the contribution of the trace term 
in the norm of a vetrex $\alpha\in\cV(p,N)$ diminishes 
towards zero as $p$ increases.
This is happening because almost all vertices are balanced, 
meaning they have approximately equal numbers of components 
equal to  $N$ or $-N$. 
The reason lies in the combinatorial property of binomial coefficients, 
which clusters their weight around their middle
(for a further in-depth discussion see the proof and an application of the argument to establish a property of equidistribution on the rays cutting through the Proth-Gilbreath triangles in the proof of~\cite[Theorem 5]{BCZ2024}).

Next, we scale the cyclotomic distance so that the points furthest apart in $\scrB(p,N)$ 
are at a distance of exactly~$1$ from each other.
This gives us the \textit{normalized distance}
\begin{align}\label{eqDeltaTrace}
  \distance_{p,N}(\alpha,\beta):=
  \sdfrac{1}{\diam(\scrB(p,N)}d(\alpha,\beta),
\end{align}
which additionally serves as a unitary means of comparing 
the spacing of points in different hypercubes~$\scrB(p,N)$.

\begin{figure}[htb]
 \centering
     \includegraphics[width=0.476\textwidth]{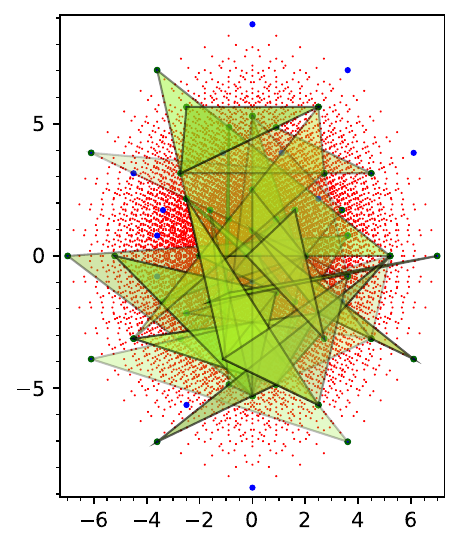}
    \includegraphics[width=0.51\textwidth]{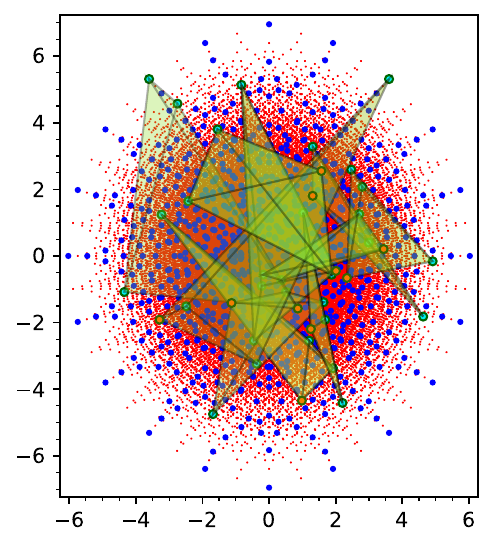}
\caption{
In the representation on the left-side, there are
twenty-six random triangles with vertices in $\cV(7,2)$, and 
on the image on the right-side there are ten pyramids with triangular bases with vertices in $\cV(11,1)$ and the apex in $\scrB(11,1)$.
The apexes of the pyramids and 
the vertices in $\cV(p,N)$ are shown 
as thicker in size compared to the rest of points in $\scrB(p,N)$,
and the vertices in $\cV(p,N)$ are also
colored differently (blue instead of red, in the electronic version).} 
\label{FigureRandomPolygons}
 \end{figure}

\subsection{The main results}\label{SubsectionMainResults}
We outline here our main findings derived from exploring topics that include: the evaluation of the distances between an arbitrary $\alpha\in\scrB(p,N)$
and the vertices $\beta\in\cV(p,N)$,
finding the common size of the central angle 
$\widehat{\alpha O\beta}$, and
measuring  distances between vertices.

First we point out that almost all triangles 
$\triangle=(\beta_1,\alpha,\beta_2)$
with $\beta_1,\beta_2\in\cV(p,N)$ 
and $\alpha$  anywhere in $\scrB(p,N)$
are almost isosceles.
\begin{theorem}\label{TheoremIsosceles}
For any $\varepsilon>0$, there exists a prime $p_\varepsilon$ such that for all primes $p\geq p_\varepsilon$,
for all integers $N\geq 1$, and for every point $\alpha\in\scrB(p,N)$, 
the proportion of triangles 
\mbox{$(\beta_1,\alpha,\beta_2)$}
such that
\begin{equation*}
  \abs{\distance_{p,N}(\alpha,\beta_1)-\distance_{p,N}(\alpha,\beta_2)}\leq\varepsilon,
\end{equation*}
where $\beta_1,\beta_2\in\cV(p,N)$,
is greater than or equal 
to $1-\varepsilon$.
\end{theorem}

Next we show that for almost all $\alpha\in\scrB(p,N)$,
nearly all angles $\widehat{\alpha O\beta}$ with 
$\beta\in\cV(p,N)$ are close to right angles.

\begin{theorem}\label{TheoremAngle}
For any $\varepsilon>0$, there exists a prime $p_\varepsilon$
and a function $f(p)\leq1/2$
that tends to zero as $p$ tends to infinity, 
such that, for all
primes $p\geq p_\varepsilon$, for all integers $N\ge 1$, 
and for all $\alpha\in\scrB(p,N)$ with 
\mbox{$\distance_{p,N}(\alpha,O)\geq f(p)$} (where
$O$ is the origin), the proportion of
angles~$\widehat{\alpha O\beta}$,
with \mbox{$\beta\in\cV(p,N)$}, for which 
\begin{equation*}
  \mid \cos (\widehat{\alpha O \beta})\mid \leq \varepsilon,
\end{equation*} 
is greater than or equal to $1-\varepsilon$.
\end{theorem}
An explicit expression for the function $f(p)$
from Theorem~\ref{TheoremAngle},
whose role is to keep triangles not too close to being degenerate, is given in 
Theorem~\ref{TheoremRightAngles}.

The following
is a super-regularity statement on the vertices of the box $\scrB(p,N)$, proving that the exceptional set of 
segments whose normalized lengths are not in the vicinity 
of~$1/\sqrt{2}$
becomes arbitrarily small as $p$ approaches infinity.
\begin{theorem}\label{TheoremMaineps}
There exist absolute positive constants $C_{1}$ and $C_2$ such that for all $\varepsilon\in (0,1)$, 
all integers $N\ge 1$, and all prime $p>C_1/\varepsilon$, 
we have
\begin{equation}\label{eqTheoremMaineps}
	\sdfrac{1}{\#\cV(p,N)^2}   
    \#\Big\{(\beta_1,\beta_2)\in\cV(p,N)^2 : 
        \Big|\distance_{p,N}(\beta_1,\beta_2)-\sdfrac{1}{\sqrt{2}}\Big|
        \ge \varepsilon\Big\}
        \le  \sdfrac{C_2}{\varepsilon^2p}\,.
\end{equation}
\end{theorem}
Note that for a non-trivial result in~\eqref{eqTheoremMaineps}, we need to take 
$p\gg 1/\varepsilon^2$.

\smallskip

By repeatedly applying Theorem~\ref{TheoremMaineps},
we derive a statement that contains it and is at the same time more general.

\begin{corollary}\label{CorollaryKeps}
There exist absolute positive constants $C_{1}$ and $C_2$ such that for any $\varepsilon\in (0,1)$, 
all integers $N\ge 1$, all integers $K\ge 2$, and
all primes $p>C_1/\varepsilon$, 
we have
\begin{equation*}
  \begin{split}
     	\sdfrac{1}{\#\cV(p,N)^{K}}   
    \#\Big\{(\beta_1,\dots,\beta_K)\in\cV(p,N)^K : &
        \max_{1\le j<k\le K}\Big|\distance_{p,N}(\beta_j,\beta_k)-\sdfrac{1}{\sqrt{2}}\Big|
 \le \varepsilon\Big\}        
        \ge  1-\sdfrac{C_{2}K^2}{\varepsilon^2 p}\,. 
  \end{split}
\end{equation*}
\end{corollary}
It should be noted that for a non-trivial result in  Corollary~\ref{CorollaryKeps}, we need to take 
$p\gg K^2/\varepsilon^2$.

Let us remark that if $p$ increases while $N$ may not necessarily change, 
the densities of both $\scrB(p,N)$ and $\cV(p,N)$, as sets of complex numbers,
increase much faster than their diameter.
In order to illustrate the resulting variety of  polytopes 
with vertices based on~$\cV(p,N)$, we show in Figure~\ref{FigureRandomPolygons} 
two such sets of $K$-polytopes with $K = 3$.
In the right-hand side figure,
the vertices of the tuples are linked to a fourth point $\alpha\in\scrB(p,N)$, forming a `pyramid' together.
Nevertheless, Corollary~\ref{CorollaryKeps} proves that 
the semblance of variety
is an appearance that changes as $p$ increases, 
and the wide variety of $K$-polytope types is completely overwhelmed by those that have all edges almost equal.
In other words, if we call \textit{super-regular} any polytope 
with all edges almost equal, then Corollary~\ref{CorollaryKeps}
says that if $K\ge 2$, then almost all $K$-polytopes with vertices in $\cV(p,N)$ are
super-regular.

\smallskip

The result in Corollary~\ref{CorollaryKeps} falls under the category of 
\textit{unit-distance} problem (see Erd\H{o}s~\cite{Erd1946}, 
Oberlin and Oberlin~\cite{OO2015} and the references therein).
Likewise, putting together distances of a kind, 
we find here the classic \textit{integer-set} problem
(see Anning and Erd\H{o}s~\cite{Erd1945, AE1945}, 
Iosevich, Shparlinski, and Xiong~\cite{ISX2010}, 
Greenfeld, Iliopoulou and Peluse~\cite{GIP2024}).
Although such sets in the $\RR^2$ plane can be infinite and intricate,
by distinction with the `almost all' result in Corollary~\ref{CorollaryKeps}, 
it is recently proven in~\cite{GIP2024} that
an integer subset of $[-N,N]^2$ can contain at most 
$O\big(\log^{O(1)}N\big)$ points that do not lie 
on a single straight line or on a circle.

\smallskip

Our results can be summarized in the following 
unified form in which we also make precise the shape of the polytopes as obtained in Theorem~\ref{TheoremKVV}.

\begin{corollary}\label{CorollaryPyramids}
Let $K\ge 2$ and $N\ge 1$ be integers and let $p$ a sufficiently large prime.
Then almost all pyramids with apex $\alpha\in\scrB(p,N)$
and $K$-polytopal bases $(\beta_1,\dots,\beta_K)\in\cV(p,N)^K$
are almost regular, and the normalized size of the 
edges and diagonals of the bases are all almost equal 
with $1/\sqrt{2}$.
Additionally, if the apex of the pyramid is close to the origin, then all lateral faces 
are almost isosceles triangles and also nearly right-angled triangles with edges $1/2,1/2,1/\sqrt{2}$.
\end{corollary}

\subsection{Pyramids with leveled heights}
Let us now level the pyramids by allowing the vertices of the bases to no longer be confined in the set $\cV(p,N)$.
Now, in the significantly larger set of all points within $\scrB(p,N)$,
a result, which can be proven by following the analogous steps in the proof of Corollary~\ref{CorollaryKeps}, 
takes place. 
The only difference in the result lies in the fact that the pyramids mentioned in Corollary~\ref{CorollaryPyramids} no longer have a distinguished vertex as the apex. 
Thus, all the lateral faces of the pyramids are nearly identical, being almost equilateral triangles, just like all triangles formed by the vertices of the bases. Additionally, 
and here lies the key distinction,
the normalized distance between almost any two points 
in~$\scrB(p,N)$
is about $1/\sqrt{6}\approx 0.408$ instead of $1/\sqrt{2}\approx 0.707$, as it was in the case of triangles with vertices in $\cV(p,N)$.

It is noteworthy that this $1/\sqrt{6}$ is a typical occurrence in the context of normalized distances in various multidimensional settings (see~\cite{ACZ2023, ACZ2024}).
A particular instance is related to the phenomenon of visibility.
Extensive research has been conducted on the various visibility problems
in sets of lattice points in different dimensions, when viewed from a fixed point 
or while walking along curves. For a limited selection of some of the relatively 
recent works that are closer to our subject see~Anderson et al.~\cite{ABCZ2024a, ABCZ2024b},
Boca et al.~\cite{BCZ2000,BS2022,BZ2005},
Athreya et al.~\cite{ACZ2023}, Lu and Meng~\cite{LM2023,Lu2024}.

Given two points in the cyclotomic hypercube $\scrB(p,N)$, we say that $\beta$ 
is \textit{visible} from $\alpha$ if their associated vectors in $\ZZ^{p-1}$ have no other
associated vector to some $\gamma\in\scrB(p,N)$
lying on the same straight line between them. 
Then, a $K$-polytope with vertices in $\scrB(p,N)$
is called \textit{self-visible} if all its vertices are visible to each other.

For any integer $K\ge 2$, let $\Omega_K(p,N)$ denote the set of all self-visible 
$K$-polytopes with vertices in $\scrB(p,N)$.
Then, knowing that, in the limit as $p$ tends to infinity, 
the set of self-visible $K$-polytopes 
is a large set, having a positive proportion in $\scrB(p,N)$  
(see~\mbox{\cite[Theorem 2]{ACZ2023}}), the next result
proves that almost all self-visible $K$-polytopes are super-regular.
\begin{theorem}\label{TheoremKVisibilityeps}
For any $\varepsilon>0$
and any integer $K\ge 2$,
there exists an effectively computable 
constant $C_{\varepsilon,K}$ such that for any prime
$p>C_{\varepsilon,K}$, and any positive integer
$N$ satisfying 
\mbox{$N/p\ge C_{\varepsilon,K}$}, we have
\begin{equation*}
	\frac{1}{\#\Omega_K(p,N)}\#\left\{(\alpha_1,\dots,\alpha_K)
 \in \Omega_K(p,N):\max_{1\le j<k\le K}\Big|\mathfrak{d}_{p,N}(\alpha_j,\alpha_k)
 -\sdfrac{1}{\sqrt{6}}\right|\le \varepsilon\Big\}\ge 1-\varepsilon.
\end{equation*}
\end{theorem}

\medskip

The proofs of our results rely on calculating the distance moments, 
for which we obtain exact formulas in Sections~\ref{SectionAMalphaV} and ~\ref{SectionAMVV}. 
By the end, in Section~\ref{SectionProofsOfTheorems}, 
the conclusions are stated in an intermediate but more explicit form 
that allows for the tracing of the constants involved.

\section{The norm, the distance and the normalized distance}\label{SectionTheDistance}
Cyclotomic fields have interesting properties that make them 
useful in the study of algebraic integers, class field theory, the theory of modular forms, while also having ties to other areas of mathematics such as algebraic geometry and representation theory. 
In this section we will address only 
the essential topics needed to reach our objectives, while for other 
basic or advanced elements, we refer to the works 
by Marcus~\cite{Mar2018} and Washington~\cite{Was1997}.

Let $p$ be an odd prime and let $\omega=\exp(2\pi i/p)$ be a primitive $p$-th root of unity.
The cyclotomic field $\QQ(\omega)$ is a subfield of complex numbers obtained by adjoining $\omega$ to the field of rational numbers.
The permutations of any $\alpha\in\QQ(\omega)$
given by the automorphisms in the \mbox{Galois} group
$Gal(\QQ(\omega)/\QQ)$, which is isomorphic with the  
multiplicative group $(\ZZ/p\ZZ)^{\times}$, 
are weighed by the trace $\Tr_{\QQ(\omega)/\QQ}=\Tr$
defined as the additive map 
$\Tr_{}: \QQ(\omega) \rightarrow \QQ$ 
and  
\begin{equation}\label{eqDefTr}
   \Tr(\alpha) = \sum_{\sigma \in Gal(\QQ(\omega)/\QQ)} \sigma(\alpha).
\end{equation}
In particular, if $\alpha\in\ZZ[\omega]=\{a_1\omega+\cdots+a_{p-1}\omega^{p-1} : a_j\in\ZZ\}$, 
which is the ring of integers of~$\QQ(\omega)$, then $\Tr(\alpha)\in\ZZ$.
\subsection{The norm}
Let $\omega,\omega^2,\dots,\omega^{p-1}$
be a basis in $\QQ(\omega)$. 
Then any $\alpha\in\QQ(\omega)$ can 
be mirrored in the $(p-1)$-dimensional vector space
$\QQ^{p-1}$ by the map
$\psi: \QQ(\omega) \rightarrow\QQ^{p-1}$
that associates the trace 
coordinates 
\begin{equation*}
    \psi(\alpha) := \big(
    \Tr(\alpha\omega), \Tr(\alpha\omega^2),\dots,
    \Tr(\alpha\omega^{p-1})
    \big).
\end{equation*}
This linking induces naturally in $\QQ(\omega)$ a \textit{norm}
through the Euclidean norm of $\psi(\alpha)\in\QQ^{p-1}$.
Thus, we define
\begin{equation}\label{eqDefNorm}
    \nnorm{\alpha} :=
    \Big(
    \Tr(\alpha\omega)^2+\Tr(\alpha\omega^2)^2+\cdots+
    \Tr(\alpha\omega^{p-1})^2
    \Big)^{1/2}.
\end{equation}
To ascertain that $\nnorm{\cdot}$ defined by~\eqref{eqDefNorm} is a norm, 
we need to check that it is positive definite.
To start with, we observe that $\nnorm{\alpha}$ is always non-negative and $\nnorm{0}=0$.

Now suppose that $\nnorm{\alpha}=0$ and $\alpha\neq 0$.
Then, on one hand, $\psi(\alpha)=0$, which means that 
$\Tr(\alpha\omega) =\Tr(\alpha\omega^2)=\cdots=
    \Tr(\alpha\omega^{p-1})=0$.
On the other hand, writing $\alpha^{-1}$ in the 
$\omega$-basis as $\alpha^{-1}=b_1\omega+\cdots+b_{p-1}\omega^{p-1}$ 
with rational coefficients $b_1,\dots,b_{p-1}$, it follows that
$b_1\alpha\omega+\cdots+b_{p-1}\alpha\omega^{p-1}=1$.
Taking trace on both sides, this implies the contradiction
\begin{equation*}
  0 = b_1\Tr(\alpha\omega)+\cdots 
     + b_{p-1}\Tr(\alpha\omega^{p-1})
  =\Tr(1) = p-1.   
\end{equation*}
In conclusion, $\nnorm{\cdot}$ is positive definite.

Next, the homogeneity 
$\nnorm{b\alpha} = |b|\nnorm{\alpha}$
holds for any $b\in\QQ$ and $\alpha\in\QQ(\omega)$
because the trace is $\QQ$-linear.

Finally, the triangle inequality for $\nnorm{\cdot}$
is a consequence of the triangle inequality in 
$\QQ^{p-1}$.

Consequently, we have ascertained that $\nnorm{\cdot}$ 
defined by~\eqref{eqDefNorm} is indeed a norm.

\subsection{The distance}
The key outcome of the previous remarks is that $\QQ(\omega)$ becomes a metric space with the \textit{norm-induced distance} 
\begin{equation}
    d(\alpha,\beta):=\nnorm{\beta-\alpha},
\end{equation}
which satisfies the required axioms, as the norm provides the necessary properties.

It is worth noting that the distance defined in this manner 
closely resembles the Euclidean distance in vector spaces and also 
has properties that are well-suited to the characteristics of cyclotomic fields.
We mention here only two of them, which will be further explored in an upcoming paper.

Thus one can show that it is invariant under the action
of the Galois group 
$\Gal(\QQ(\omega)/\QQ)$, that is,
$d(\alpha,\beta)=d\big(\sigma(\alpha),\sigma(\beta)\big)$
for any automorphism $\sigma\in \Gal(\QQ(\omega)/\QQ)$.

Another noted property establishes a connection between 
the distance between two points and the layout of the entire fields they generate. 
More precisely, the property that is reminiscent of 
Krasner's fundamental Lemma~\mbox{\cite[{Lemma 8.1.6}]{NW2008})}
states that if $\alpha$ and $\beta$ are two elements in $\QQ(\omega)$ 
and if $\alpha$ is closer to $\beta$  
than half the distance to its closest away conjugate, then the entire 
field generated by $\alpha$ is included in the field generated by $\beta$.

\subsection{The central hypercube and the Euclidean norm}
For any odd prime $p$ and positive integer $N$, denote
$\omega=\exp(2\pi i/p)$ and let
$\scrB(p,N)$ be the symmetric hypercube defined by
\begin{equation*}
   \scrB(p,N):=\left\{a_1\omega+\cdots+a_{p-1}\omega^{p-1} \in\CC: 
    (a_1,\dots,a_{p-1})\in \big([-N,N]\cap  \ZZ\big)^{p-1}\right\}
    \subset\ZZ[\omega].
\end{equation*}
Because $\{\omega, \omega^2,\dots,\omega^{p-1}\}$ is a basis in $\ZZ[\omega]$, 
we know that ${\scrB}(p,N)$ contains 
\begin{equation}\label{eqCardinalB}
    \#{\scrB}(p,N) = (2N+1)^{p-1}
\end{equation}
distinct elements.

For any $\alpha = a_1\omega+\cdots+a_{p-1}\omega^{p-1} \in\QQ(\omega)$, 
we always write its associated vector of coordinates
$\balpha = (a_1,\dots,a_{p-1})\in\QQ^{p-1}$ 
in boldface without further  additional repetitions.
Then, for any $\alpha \in\QQ(\omega)$, we denote by $\nnorm{\balpha}_E$
the \textit{Euclidean norm} of $\alpha$, which is 
the Euclidean norm of $\balpha$ defined by 
\begin{equation}\label{eqEuclideanNorm}
   \nnorm{\balpha}_E:= \bigg(\sum_{i=1}^{p-1}a_i^2\bigg)^{1/2}.  
\end{equation}

\subsection{The key formula for the norm}
From the definition of the trace~\eqref{eqDefTr} it follows that
$\Tr(1)=p-1$,  and $\Tr(\omega)=\cdots=\Tr(\omega^{p-1})=-1$.
These imply two simple formulas for the trace. Thus, if
$\alpha=a_1\omega+\cdots+a_{p-1}\omega^{p-1}$, then
\begin{equation}\label{eqTrSum}
  \Tr(\alpha)=-(a_1+\cdots+a_{p-1}),
\end{equation}
and
\begin{equation}\label{eqTr2}
  \Tr(\alpha\omega^j)
  = \sum_{\substack{k=1\\k\ne p-j}}^{p-1}(-a_k)+(p-1)a_{p-j}
  =\Tr(\alpha)+pa_{p-j},
  \ \ \text{ for $j=1,\dots,p-1$.}
\end{equation}
Then, by combining definition~\eqref{eqDefNorm} and
relation~\eqref{eqTr2}, we see that
\begin{equation*}
    \nnorm{\alpha}^2 = \sum_{j=1}^{p-1}\Tr(\alpha\omega^j)^2
	=\sum_{j=1}^{p-1}\left(\Tr(\alpha)+pa_{p-j}\right)^2
	=\sum_{j=1}^{p-1}\left(\Tr(\alpha)+pa_{j}\right)^2.
\end{equation*}
Next, by expanding the square, employing~\eqref{eqTrSum} and 
simplifying the expression, it yields:
\begin{equation*}
  \begin{split}
      \nnorm{\alpha}^2 
      &=\sum_{j=1}^{p-1}\left(\Tr(\alpha)^2+2pa_j\Tr(\alpha)+p^2a_j^2\right)\\
	&=(p-1)\Tr(\alpha)^2+2p\Tr(\alpha)\sum_{j=1}^{p-1}a_j+p^2\sum_{j=1}^{p-1}a_j^2\\
	&=(p-1)\Tr(\alpha)^2-2p\Tr(\alpha)^2+p^2\sum_{j=1}^{p-1}a_j^2\\
	&=p^2\sum_{j=1}^{p-1}a_j^2-(p+1)\Tr(\alpha)^2.
  \end{split}
\end{equation*}
Thus, we have obtained a convenient formula that we state 
in the next lemma for later reference.
\begin{lemma}\label{LemmaNormFormula}
Let $\alpha=a_1\omega+\cdots+a_{p-1}\omega^{p-1}\in\QQ(\omega)$, 
and denote by $\balpha$ the associated 
vector $\balpha=(a_1,\dots,a_{p-1})\in\QQ^{p-1}$.
Then
\begin{equation}\label{eqNorm_Formula}
   \nnorm{\alpha}^2=p^2\nnorm{\balpha}_E^2-(p+1)\Tr(\alpha)^2.
\end{equation}
\end{lemma}
As a point of reference, let us see the relation of size between the norms
$\nnorm{\cdot}_E$ and $\nnorm{\cdot}$.
\begin{remark}
For any $\alpha\in\QQ(\omega)$, 
with the associated vector $\balpha\in\QQ^{p-1}$,   
  we have
\begin{equation}\label{eqNormsInequality}
    \nnorm{\balpha}_E \le \nnorm{\alpha}.
\end{equation}
\end{remark}
To prove~\eqref{eqNormsInequality}, let
$\alpha=a_1\omega+\cdots+a_{p-1}\omega^{p-1}$ and appeal to the Cauchy-Schwartz inequality to get
\begin{equation*}
   \Tr(\alpha)^2=\bigg(\sum_{j=1}^{p-1}a_j\bigg)^2
     \le(p-1)\sum_{j=1}^{p-1}a_j^2
     = (p-1)\nnorm{\balpha}^2_E.
\end{equation*}
Then insert this inequality in~\eqref{eqNorm_Formula} to find that
\begin{equation*}
    \nnorm{\alpha}^2
    = p^2\nnorm{\balpha}_E^2-(p+1)\Tr(\alpha)^2
    \ge p^2\nnorm{\balpha}_E^2 - (p+1)(p-1)\nnorm{\balpha}_E^2
    =\nnorm{\balpha}_E^2,
\end{equation*}
from which~\eqref{eqNormsInequality} follows.

\subsection{The Euclidean diameter of the box}
Let $N\ge 0$ and $q\ge 3$ be integers. Let $\omega=\exp(2\pi i/q)$ be the fundamental $q$-th root of unity. 
(Our results are specifically for the case where $q$ is an odd
prime, and as the concepts self-support for a general $q\ge 3$, a more comprehensive discussion is needed elsewhere.)

Let $\scrB(q,N)\subset\mathbb{C}$ denote the set
of complex numbers represented as $f(\omega)$,
where $f\in\mathbb{Z}[X]$, $f(X)=a_1X+\cdots+a_{q-1}X^{q-1}$ and the coefficients of $f$ are in $[-N,N]$, that is,
\begin{equation*}
    \scrB(q,N) := \big\{
    z\in\CC : z = a_1\omega+\cdots+a_{q-1}\omega^{q-1},\
    a_1,\dots,a_{q-1}\in [-N,N]
    \big\}.
\end{equation*}
Then the set of vertices  $\mathcal{V}_{q,N}\subset \scrB(q,N)$ is obtained
when the coefficients take values only at the endpoints of the interval:
\begin{equation*}
    \cV(q,N) := \big\{
    z\in\mathbb{C} : z = a_1\omega+\cdots+a_{q-1}\omega^{q-1},\
    a_1,\dots,a_{q-1}\in \{-N,N\}
    \big\}.
\end{equation*}
Note that the geometrical shape of $\cV(q,N)$ is not changed when scaling by $N$, so that we may assume $N=1$.
\begin{figure}[htb]
 \centering
 \hfill
     \includegraphics[width=0.24\textwidth]{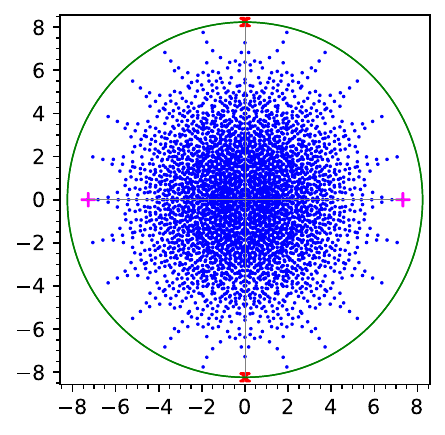}
    \includegraphics[width=0.24\textwidth]{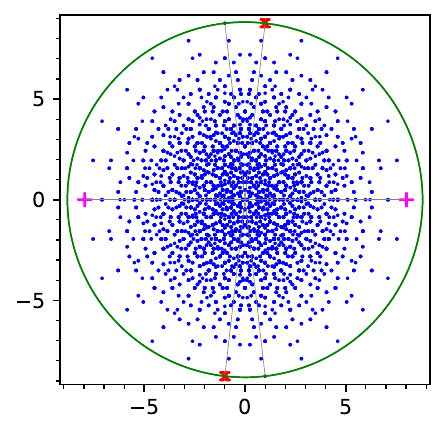}
     \includegraphics[width=0.24\textwidth]{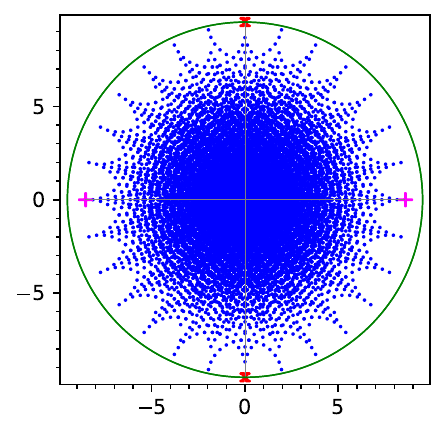}
    \includegraphics[width=0.246\textwidth]{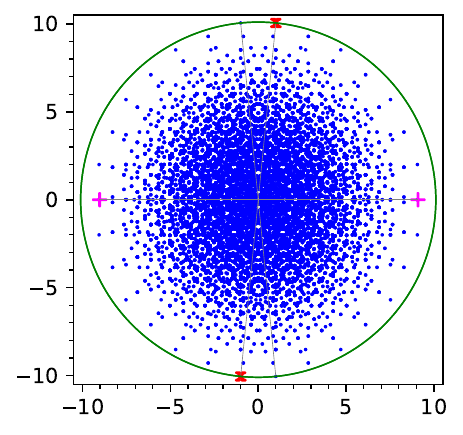}
\caption{
The set of vertices $\cV{(q,1)}$, 
their encircling circle and 
the poles highlighted for $q=13,14,15,16$.
} 
\label{FigurePoles13-16}
 \end{figure}
Two significant points in the box are the vertices North Pole 
and the East Pole, which we denote by $\NP(q)$ and $\EP(q)$.
While $\EP(q)$ is real for every $q$, $\NP(q)$ is purely imaginary only if $q$ is odd.
When $q$ is even, $\NP(q)$ is placed in the first quadrant closest to the imaginary axis among the vertices above the real axis farthest from the origin, and has a mirror double across the imaginary axis.
The set of vertices and the poles are shown in Figure~\ref{FigurePoles13-16} for $q=13,14,15,16$.
The southern and western poles $\SP(q)$ and $\WP(q)$ are the opposites of $\NP(q)$ and $\EP(q)$ and are their
mirror reflections with respect to the origin.

The North Pole is the purely imaginary complex number:
\begin{equation}\label{eqNorthPole}
 \NP(q)= \sum_{j=1}^{(q-1)/2}\omega^j
 -\sum_{j=(q+1)/2}^{q-1}\omega^j.
\end{equation}

The East Pole is the real number:
\begin{equation}\label{eqEastPole}
 \EP(q)= \sum_{j=1}^{\lfloor{q/4}\rfloor}\omega^j-\sum_{j=\lfloor{q/4}\rfloor+1}^{\lfloor{3q/4}\rfloor}\omega^j
 \sum_{j=\lfloor{3q/4}\rfloor+1}^{q-1}\omega^j.
\end{equation}

Then the Euclidean diameter of $\scrB(q,N)$ is
\begin{equation}\label{eqBEuclideanDiameter}
 \diam\nolimits_E(\scrB(q,N)) 
 = N\cdot \nnorm{ \NP(q)-\big(-\NP(q)\big)}_E = \sdfrac{2N}{i}\cdot \NP(q),\ \text{ if $q$ is odd}.
\end{equation} 

\subsection{The normalized distance}
To ensure a unified reference for comparison from hypercube to hypercube as $N$ or $p$
change, we will use in each hypercube  $\scrB(p,N)$
a normalized distance $\fd_{p,N}(\alpha,\beta)$, 
which depends on both $N$ and~$p$, but has the property that the diameter of 
$\scrB(p,N)$ measured by $\fd_{p,N}$ is equal to $1$.
In order to do this, we need to know the diameter of $\scrB(p,N)$, 
which we will calculate in the following lemma.

Let $ \diam\big(\scrB(p,N)\big)$ denote the \textit{diameter} of $\scrB(p,N)$,
which is defined as the maximum distance between any two points in the hypercube,
that is,
\begin{equation*}
    \diam\big(\scrB(p,N)\big) := \max_{\alpha,\beta\in\scrB(p,N)}
    \nnorm{\beta-\alpha}\,.
\end{equation*}

\begin{lemma}\label{LemmaDiameter}
Let $p$ be an odd prime and let $N$ be a positive integer. Then
\begin{equation}\label{eqLemmaDiameter}
    \diam\big(\scrB(p,N)\big) = 2Np(p-1)^{1/2}.
\end{equation}    
\end{lemma}
\begin{proof}
From the expression~\eqref{eqNorm_Formula}  of the norm, we see that 
 $\alpha$ and $\beta$ are diametrically opposite points in $\scrB(p,N)$ 
 if they are in a certain position that makes 
 $p^2\nnorm{\bbeta-\balpha}_E^2-(p+1)\Tr(\alpha)^2$ as large as possible,
 where $\balpha$ and $\bbeta$ are the coordinate vectors associated 
 with  $\alpha$ and $\beta$.

 Next let us note that 
 \begin{equation}\label{eq2inequalities}
   \nnorm{\bbeta-\balpha}_E^2\le (p-1)(2N)^{2}\ \text{ and }\  
   \big(\Tr(\beta-\alpha)\big)^2\ge 0
 \end{equation}
for all $\alpha,\beta\in\scrB(p,N)$.

But, if we take $\alpha$ and $\beta$, for example, with alternating
coefficients, such as
\begin{equation*}
      \alpha_0 := \sum_{j=1}^{p-1} (-1)^jN\omega^j\ \text{ and }\  
      \beta_0 := \sum_{j=1}^{p-1} (-1)^{j+1}N\omega^j,
\end{equation*}
we see that $\beta_0 = -\alpha_0$ and $\Tr(\alpha_0)=\Tr(\beta_0)=0$,
because $p$ is odd.
Then, \mbox{$\Tr(\beta_0-\alpha_0) = 0$} and also
we have 
$\nnorm{\bbeta_0-\balpha_0}_E^2 = (p-1)(2N)^{2}$, 
that is, both inequalities in~\eqref{eq2inequalities} are satisfied 
with equalities.
As a consequence,
\begin{equation*}
    \diam\Big(\scrB(p,N)\Big) = \nnorm{\beta_0-\alpha_0} =
    \big(p^2\cdot(p-1) (2N)^2\big)^{1/2} = 2Np(p-1)^{1/2},
\end{equation*}
which concludes the proof of the lemma.
\end{proof}
Now, taking into account formula~\eqref{eqLemmaDiameter}, we can define 
the normalized measure that allows for a standardized comparison of 
the distances in different boxes $\scrB(p,N)$ when $N$ and $p$ vary.
 Thus, for any $\alpha,\beta\in\scrB(p,N)$, we define the \textit{normalized distance} 
 $\distance_{p,N}(\alpha,\beta)$ by
\begin{equation}\label{eqNormalizedDistance}
    \distance_{p,N}(\alpha,\beta)
    :=\frac{1}{2Np(p-1)^{1/2}} d(\alpha,\beta).
\end{equation}
Note that, while this definition guarantees that the normalized distance
between any two points $\alpha$ and~$\beta$ in $\scrB(p,N)$ is at most $1$, the formula can also be used for any pair of points  $\alpha,\beta\in\ZZ[\omega]$, even if their mutual normalized distance may be greater than $1$.
Also, from the proof of Lemma~\ref{LemmaDiameter} we may say that
the normalized diameter of both $\scrB(p,N)$
and $\cV(p,N)$ is equal to $1$.
\section{On the distances from an arbitrary point to the vertices}\label{SectionAMalphaV}
\subsection{The average distance from a point to the vertices}
Let $A_{p,N}(\alpha,\cV)$ denote the average normalized 
square distances from $\alpha\in\ZZ[\omega]$ to the vertices of $\scrB(p,N)$. 
The next lemma provides an exact formula for 
$A_{p,N}(\alpha,\cV)$.
\begin{lemma}\label{LemmaAverage}
  Let $p$ be an odd prime, let $N\ge 1$ be integer
  and let $\cV$ be the set of the vertices of the origin-centered hypercube $\scrB(p,N)$.
Then, for any $\alpha\in\ZZ[\omega]$, the average of the 
square normalized distance from $\alpha$ to $\cV$ is
\begin{equation}\label{eqAverageLA}
 \begin{split}
 A_{p,N}(\alpha,\cV)
   & := \frac{1}{\#\cV}
     \sum_{x \in\cV}\big(\distance_{p,N}(\alpha,x)\big)^2
     = \big(\distance_{p,N}(0,\alpha)\big)^2  
   + \sdfrac{1}{4}
   - \sdfrac{1}{4}p^{-1}
   - \sdfrac{1}{4}p^{-2}\,.  
 \end{split}
\end{equation}    
\end{lemma}
Note that the average $A_{p,N}(\alpha,\cV)$ depends on the position of
$\alpha$ in $\scrB(p,N)$ but not on the size $N$ of the box.

\begin{proof}
In the $\omega$-basis, 
let $\alpha=a_1\omega+\cdots+a_{p-1}\omega^{p-1}$
and denote a generic element of $\cV$ by
 $x=x_1\omega+\cdots+x_{p-1}\omega^{p-1}$.
Correspondingly, we denote by 
$\balpha=(a_1,\dots,a_{p-1})$ and 
$\bx=(x_1,\dots,x_{p-1})\in\{-N,N\}^{p-1}$ the associated vectors.
 Let $D = 2Np(p-1)^{1/2}$ denote the diameter of $\scrB(p,N)$.
 Then, employing formula ~\eqref{eqNorm_Formula},
 we have
 \begin{equation}\label{eqApN1}
  \begin{split}
 A_{p,N}(\alpha,\cV)
   & = \frac{p^2}{\#\cV D^2}
     \sum_{x \in\cV}\sum_{j=1}^{p-1}(a_j-x_j)^2
     -\frac{p+1}{\#\cV D^2}
\sum_{x \in\cV}\sum_{j=1}^{p-1}\sum_{k=1}^{p-1}(a_j-x_j)(a_k-x_k)
     \\
   &= \Sigma_A'-\Sigma_A'',      
  \end{split}
\end{equation} 
where $\Sigma_A'$ and  $\Sigma_A''$  
stand for the minuend and the subtrahend above.

The minuend is:
\begin{equation}\label{eqAS1}
  \begin{split}
 \Sigma_A'
   & = \frac{p^2}{\#\cV D^2}\sum_{x \in\cV}
   \sum_{j=1}^{p-1}  (a_j^2-2a_jx_j +x_j^2)
     \\    &
   = \frac{p^2}{D^2}\big(\nnorm{\balpha}_E^2+(p-1)N^2
   \big),
  \end{split}
\end{equation} 
because $x_j^2=N^2$ for any $x\in\cV$ and the number of
vertices with the $j$-th component equal to $N$ or $-N$ are the same,
so that the sum of the terms with $x_j$ at the first power vanishes.

Next to calculate $\Sigma_A''$, note that 
using the same fact on the cancellation in the sum 
of the terms with components at the first power, we have
\begin{equation}\label{eqAS21}
  \begin{split}
\sum_{x \in\cV}\sum_{j=1}^{p-1}\sum_{k=1}^{p-1}(a_j-x_j)(a_k-x_k) & = 
\sum_{x \in\cV}\bigg(\sum_{j=1}^{p-1}a_j\bigg)^2
-2 \sum_{j=1}^{p-1}a_j\sum_{x \in\cV}\sum_{k=1}^{p-1}x_k
+\sum_{x \in\cV}\sum_{j=1}^{p-1}\sum_{k=1}^{p-1}x_jx_k\\
& = \#\cV\Tr(\alpha)^2
   +\sum_{x \in\cV}\sum_{j=1}^{p-1}\sum_{k=1}^{p-1}x_jx_k.
  \end{split}   
\end{equation} 
In the last multiple sum, we separate the terms on the diagonal from the off-diagonal ones, to obtain:
\begin{equation}\label{eqAS22}
  \begin{split}
 \sum_{x \in\cV}\sum_{j=1}^{p-1}\sum_{k=1}^{p-1}x_jx_k
 =  \sum_{j=1}^{p-1}\sum_{x \in\cV}x_j^2
 + \sum_{j=1}^{p-1}
 \sum_{\substack{k=1\\ k\neq j}}^{p-1} \sum_{x \in\cV}x_jx_k
 = \#\cV\cdot (p-1)N^2. 
  \end{split}
\end{equation}
(Here, to see that the sum of the off-diagonal terms is zero, note on one hand that for each~$i,j$ fixed, there are four possible values for the pairs of components $x_j,x_k$, namely $-N,-N$; $-N,N$; $N,-N$; $N,N$,
and in two of these instances the product $x_jx_k$ is~$N^2$,
while in the other two the product $x_jx_k$ is~$-N^2$.  
On the other hand, the number of vertices $x\in\cV$ with the components $x_j,x_k$ being in any of these four instances is the same.)

Thus, using~\eqref{eqAS21} and~\eqref{eqAS22}, we obtain
 \begin{equation}\label{eqAS2}
  \begin{split}
 \Sigma_A''
   & = \frac{1}{D^2}(p+1)\Tr(\alpha)^2
   +\frac{1}{D^2}(p-1)(p+1)N^2.
  \end{split}
\end{equation} 
Inserting both~\eqref{eqAS1} and~\eqref{eqAS2} into~\eqref{eqApN1} yields
\begin{equation}\label{eqAverageLA2}
 \begin{split}
 A_{p,N}(\alpha,\cV)
   & = \frac{1}{\left(\diam\big(\scrB(p,N)\big)\right)^2}\big(
    p^2\nnorm{\balpha}_E^2 -(p+1)\Tr(\alpha)^2+N ^2(p^3-2p^2+1)
   \big).   
 \end{split}
\end{equation} 
On using~\eqref{eqNorm_Formula} and~\eqref{eqNormalizedDistance},
formula~\eqref{eqAverageLA2}
can be rewritten as~\eqref{eqAverageLA}, 
thus completing the proof of Lemma~\ref{LemmaAverage}.
\end{proof}

\subsection{A moment about the mean}
Here we calculate the second moment 
$M_{p,N}(\alpha,\cV)$
of the square normalized
distance from $\alpha\in\ZZ[\omega]$ to $\cV(p,N)$ about the average $A_{p,N}(\alpha,\cV)$.
\begin{lemma}\label{Lemma2Moment}
  Let $p$ be an odd prime, let $N\ge 1$ be integer,
  let $\cV$ be the set of the vertices of the 
 origin-centered hypercube $\scrB(p,N)$,
 and let $D=\diam\big(\scrB(p,N)\big)$.
Then, for all~$\alpha\in\ZZ[\omega]$, the second moment 
about the mean of the square normalized distance from $\alpha$ 
to~$\cV$ is
\begin{equation}\label{eqAverageL2M}
  \begin{split}
 M_{p,N}(\alpha,\cV)
   & := \frac{1}{\#\cV}
     \sum_{x \in\cV}
     \big(\distance_{p,N}^2(\alpha,x)-A_{p,N}(\alpha,\cV)\big)^2
  \\
   &\phantom{:}= \frac{ 2N^2}{D^4}
   \Big((N^2 + 2\nnorm{\balpha}_E^2)p^4
 - (N^2 + 2\Tr(\alpha)^2)p^3 
 - (3N^2 + 2\Tr(\alpha)^2)p^2 \\
  & \phantom{= \frac{ 2N^2}{D^4}
   \Big((N^2 + 2\nnorm{\balpha}_E^2)p^4
   - (N^2 .}
 + (N^2 - 2\Tr(\alpha)^2)p 
 + 2(N^2 - \Tr(\alpha)^2)\Big).
  \end{split}
\end{equation} 

\end{lemma}
Note that since 
\begin{align*}
   \frac{ 2N^2}{D^4} = \sdfrac{1}{8N^2p^4(p-1)^2}
   = \sdfrac{1}{8N^2p^6} + O\big(N^{-2}p^{-7}\big),
\end{align*}
while $\nnorm{\balpha}_E^2\le N^2p$
and $\Tr(\alpha)^2\le N^2p^2$
it follows that $M_{p,N}(\alpha,\cV)$ is $O(p^{-1})$, 
more precisely
\begin{align}\label{eqMbound}
  0\le M_{p,N}(\alpha,\cV) \le \sdfrac{3}{p}
\end{align}
for all integer $N\ge 1$ and all odd primes $p$.
\begin{proof}
Let $\alpha=a_1\omega+\cdots+a_{p-1}\omega^{p-1}$ be fixed,
let $x=x_1\omega+\cdots+x_{p-1}\omega^{p-1}$ denote the elements of $\cV$,
and let $\balpha=(a_1,\dots,a_{p-1})\in\ZZ[\omega]$, and
$\bx=(x_1,\dots,x_{p-1})\in\{-N,N\}^{p-1}$ be the associated vectors.

Expanding the square, the second moment can be rewritten as
\begin{equation}\label{eq2M1}
  \begin{split}
 M_{p,N}(\alpha,\cV)
   & = \frac{1}{\#\cV}
     \sum_{x \in\cV} \distance_{p,N}^4(\alpha,x)
     -2 A_{p,N}(\alpha,\cV)\cdot \frac{1}{\#\cV}
     \sum_{x \in\cV} \distance_{p,N}^2(\alpha,x)
     + \big(A_{p,N}(\alpha,\cV)\big)^2
  \\
   &  = \frac{1}{\#\cV} \sum_{x \in\cV} \distance_{p,N}^4(\alpha,x)
   - \big(A_{p,N}(\alpha,\cV)\big)^2,
  \end{split}
\end{equation}
where $D = \diam\big(\scrB(p,N)\big)$.
We employ formula~\eqref{eqNorm_Formula} to express the sum above~as
\begin{equation}\label{eqMd4}
  \begin{split}
 \sum_{x \in\cV} \distance_{p,N}^4(\alpha,x) 
 & = \sdfrac{1}{D^4}\sum_{x \in\cV}\nnorm{\alpha-x}^4\\
 & = \sdfrac{1}{D^4}\sum_{x \in\cV}
 \Big(p^2\nnorm{\balpha-\bx}_E^2-(p+1)\Tr(\alpha-x)^2\Big)^2\\
 & = \Sigma_\alpha' + \Sigma_\alpha'' +\Sigma_\alpha''',
  \end{split}
\end{equation}
where
\begin{equation}\label{eqMS123}
  \begin{split}
 \Sigma_\alpha' & = \sdfrac{p^4}{D^4}\sum_{x \in\cV}\nnorm{\balpha-\bx}_E^4\\
 \Sigma_\alpha'' & = -\sdfrac{2p^2(p+1)}{D^4}\sum_{x \in\cV}\nnorm{\balpha-\bx}_E^2\Tr(\alpha-x)^2\\
  \Sigma_\alpha''' & = \sdfrac{(p+1)^2}{D^4}\sum_{x \in\cV}\Tr(\alpha-x)^4.
  \end{split}
\end{equation}

Next, in the following subsection, we bring out a few formulas that will be used several times afterwards.
The formulas are based on sums that cancel out due to the symmetry of the vertices and 
either the fact that the order of the added monomials is odd or that, in the case of an even order, the summation is only done over the `off-diagonal' terms.
 
 \subsubsection{Multinomial sums on vertices that cancel out}
For an arbitrary $\alpha\in\ZZ[\omega]$ we write
$\alpha=a_1\omega+\cdots+a_{p-1}\omega^{p-1}$
and let \mbox{$\balpha=(a_1,\dots,a_{p-1})$} be its associated vector.
We also let 
\mbox{$x=x_1\omega+\cdots+x_{p-1}\omega^{p-1}\in\cV$} denote a generic vertex of $\scrB(p,N)$. 

The following sum of linear terms cancels out as seen in the derivation of  relations~\eqref{eqAS1}:
\begin{align}\label{eqRemark1}
    \sum_{x\in\cV}\sum_{j=1}^{p-1}a_jx_j 
    = \sum_{j=1}^{p-1}a_j\sum_{x\in\cV}x_j =0. 
\end{align}

In the analogous quadratic form, we separate the diagonal and the off-diagonal terms. Then, since the sum of the latter cancels out, as noted while deriving formula~\eqref{eqAS22}, we obtain:
\begin{align}\label{eqRemark2}
  \sum_{x\in\cV}\sum_{j=1}^{p-1}\sum_{k=1}^{p-1}a_ja_kx_jx_k 
=  \sum_{j=1}^{p-1}a_j^2\sum_{x \in\cV}x_j^2
 + \sum_{j=1}^{p-1}a_j
 \sum_{\substack{k=1\\ k\neq j}}^{p-1}a_k 
 \sum_{x \in\cV}x_jx_k
 = \#\cV\cdot \nnorm{\balpha}_E^2 N^2,    
\end{align}    
and
\begin{equation}\label{eqRemark22}
  \begin{split}
  \sum_{x\in\cV}\sum_{j=1}^{p-1}\sum_{k=1}^{p-1}
  \sum_{m=1}^{p-1}\sum_{n=1}^{p-1}a_ja_kx_mx_n 
 &=  \sum_{j=1}^{p-1}a_j\sum_{k=1}^{p-1}a_k
   \sum_{x \in\cV}\bigg(\sum_{m=1}^{p-1}x_m^2
 + \sum_{m=1}^{p-1}
 \sum_{\substack{n=1\\ n\neq m}}^{p-1}x_mx_n\bigg)\\
 &= \#\cV\cdot \Tr(\alpha)^2 N^2(p-1). 
  \end{split}
\end{equation}

For a sum of cubic monomials in the components of $\bx$, we have:
\begin{align*}
  \sum_{x\in\cV}
  \sum_{j=1}^{p-1}\sum_{m=1}^{p-1}\sum_{n=1}^{p-1}
  x_jx_mx_n
  &=  \sum_{j=1}^{p-1}\sum_{x\in\cV} x_j^3
 +3 \sum_{j=1}^{p-1} \sum_{\substack{n=1\\ n\neq j}}^{p-1}
  \sum_{x\in\cV}x_j^2 x_n  
  + \sum_{j=1}^{p-1} \sum_{\substack{m=1\\ m\neq j}}^{p-1}
   \sum_{\substack{n=1\\ n\neq m\\ n\neq j}}^{p-1}
   \sum_{x\in\cV}x_jx_m x_n .  
\end{align*}
The first two terms on the right-side above are:
\begin{align*}
  \sum_{j=1}^{p-1}\sum_{x\in\cV} x_j^3
  &=  N^2 \sum_{j=1}^{p-1}\sum_{x\in\cV} x_j = 0\\[-4pt]
  \intertext{and}
 3 \sum_{j=1}^{p-1} \sum_{\substack{n=1\\ n\neq j}}^{p-1}
  \sum_{x\in\cV}x_j^2 x_n 
   &= 3 N^2 \sum_{j=1}^{p-1} \sum_{\substack{n=1\\ n\neq j}}^{p-1}
  \sum_{x\in\cV} x_n  =0,
\end{align*}
the equality with zero being a consequence of~\eqref{eqRemark1}.
Let us see that the third term also cancels out.
Indeed, the triple $(x_j,x_m,x_n)\in\{-N,N\}^3$ can take eight values
with two possible products $x_jx_mx_n=\pm N^3$.
Further, exactly four of these triples have the product of components equal to $-N^3$ and the other four have the product of components equal to~$N^3$. Finally note that the number of vertices $x\in\cV$
having as fixed components any of these triples is the same for each of them. Thus we have shown that
\begin{align}\label{eqRemark3}
  \sum_{x\in\cV}
  \sum_{j=1}^{p-1}\sum_{m=1}^{p-1}\sum_{n=1}^{p-1}
  x_jx_mx_n = 0. 
\end{align}

In the case of a sum involving a quartic expression, 
if we separate the terms with equal factors in monomials 
and note as above the cancellation of the partial sums 
of monomials with odd degrees, we obtain:
\begin{equation}\label{eqRemark4}
  \begin{split}
  \sum_{x\in\cV}
  \sum_{j=1}^{p-1}\sum_{k=1}^{p-1}\sum_{m=1}^{p-1}\sum_{n=1}^{p-1}
  x_jx_kx_mx_n 
  & =   \sum_{j=1}^{p-1}\sum_{x\in\cV}x_j^4 
     + 3 \sum_{j=1}^{p-1} \sum_{\substack{m=1\\ m\neq j}}^{p-1}\sum_{x\in\cV}x_j^2x_m^2\\
  & = \#\cV\cdot N^4(p-1) + 3\#\cV\cdot N^4(p-1)(p-2)\\
  & = \#\cV\cdot N^4(p-1)(3p-5). 
 \end{split}
\end{equation}
 \subsubsection{\texorpdfstring{The estimation of $\Sigma_\alpha'$}{The estimation of SigmaM'}}
The sum in $\Sigma_\alpha'$ is
\begin{align*}
  \sum_{x \in\cV}\nnorm{\balpha-\bx}_E^4
  &= \sum_{x \in\cV}\sum_{j=1}^{p-1}(a_j-x_j)^2
                   \sum_{k=1}^{p-1}(a_k-x_k)^2\\
  &= \sum_{x \in\cV}\sum_{j=1}^{p-1}(a_j^2-2a_jx_j+x_j^2)
                   \sum_{k=1}^{p-1}(a_k^2-2a_kx_k+x_k^2)\\
  &= \sum_{x \in\cV}
  \bigg(\nnorm{\balpha}_E^2+N^2(p-1)-2\sum_{j=1}^{p-1}a_jx_j\bigg)   \bigg(\nnorm{\balpha}_E^2+N^2(p-1)-2\sum_{k=1}^{p-1}a_kx_k\bigg).
\end{align*}
Changing the order of summation and using~\eqref{eqRemark1} yields:
\begin{align*}
  \sum_{x \in\cV}\nnorm{\balpha-\bx}_E^4
  &= \#\cV\big(\nnorm{\balpha}_E^4+2N^2(p-1)\nnorm{\balpha}_E^2
  + N^4 (p-1)^2\big)
  + 4 \sum_{x \in\cV}\sum_{j=1}^{p-1}\sum_{k=1}^{p-1}
  a_ja_kx_jx_k.
\end{align*}
On inserting here the value of the multiple sum from~\eqref{eqRemark2}, we deduce:
\begin{align}\label{eqSigmaM1}                     
    \Sigma_\alpha' = \sdfrac{\#\cV p^4}{D^4}
    \Big(\nnorm{\balpha}_E^4
    + 2N^2(p+1)\nnorm{\balpha}_E^2
  + N^4 (p-1)^2\Big).
\end{align}

 \subsubsection{\texorpdfstring{The estimation of $\Sigma_\alpha''$}{The estimation of SigmaM''}}
The sum in the definition of $\Sigma_\alpha''$ is:
\begin{align*}
    &\quad\sum_{x \in\cV}\nnorm{\balpha-\bx}_E^2
    \Tr(\alpha-x)^2
    =  \sum_{x \in\cV} \sum_{j=1}^{p-1}(a_j-x_j)^2
    \sum_{m=1}^{p-1}\sum_{n=1}^{p-1}
    (a_m-x_m)(a_n-x_n)\\
    &=  \sum_{x \in\cV} \sum_{j=1}^{p-1} 
    \sum_{m=1}^{p-1}\sum_{n=1}^{p-1}
    \big(a_j^2-2a_jx_j+x_j^2\big)
    \big(a_ma_n-a_mx_n-x_ma_n+x_mx_n\big).
\end{align*}

Expanding the product, we obtain the following twelve monomials:
\begin{align*}
  &a_j^2a_ma_n, && \uline{-a_j^2a_mx_n}, && \uline{-a_j^2x_ma_n}, && \uuline{a_j^2x_mx_n},
&& \uline{-2a_jx_ja_ma_n}, && \uuline{2a_jx_ja_mx_n},\\
  &x_j^2a_ma_n, && \dotuline{-x_j^2a_mx_n}, 
  && \dotuline{-x_j^2x_ma_n}, 
  && \uwave{x_j^2x_mx_n},
&& \uuline{2a_jx_jx_ma_n}, && \dashuline{-2a_jx_jx_mx_n}.
\end{align*}
Next, changing the order of summation, we have to sum each of these monomials over $x\in\cV$.

Note that the $1$st and the $7$th monomials are constant while $x$ runs over $\cV$.
Next, the sums of the $2$nd, the $3$rd and the $5$th monomials is zero, as seen in relation~\eqref{eqRemark1}.
For the $4$th, the $6$th and the $11$th, we separate the diagonal and the off-diagonal terms as seen in relation~\eqref{eqRemark2}.
We also use relation~\eqref{eqRemark2} to find the sum of the 
monomials of the $10$th type, noting that $x_j^2=N^2$.
The sums of the monomials of the $8$th, the $9$th cancel out
by~\eqref{eqRemark1} since $x_j^2=N^2$.
Also, the sum of the monomials of the $12$th type cancels out also, as seen in~\eqref{eqRemark3}.
These imply:
\begin{align*}
    \quad\sum_{x \in\cV}\nnorm{\balpha-\bx}_E^2\Tr(\alpha-x)^2
    & = \#\cV
       \Big(\nnorm{\balpha}_E^2 \Tr(\alpha)^2 + 0 + 0 \\
       & \phantom{ = \#\cV\big(} 
       + \nnorm{\balpha}_E^2 N^2 (p-1)  + 0 + 2 \Tr(\alpha)^2N^2 \\
        & \phantom{ = \#\cV\big(} 
        + N^2(p-1)\Tr(\alpha)^2 + 0  + 0 \\
        & \phantom{ = \#\cV\big(} 
        +\big(N^4(p-1) + N^4(p-1)(p-2)\big) + 2\Tr(\alpha)^2N^2 + 0
       \Big).
\end{align*}
Therefore:
\begin{equation}\label{eqSigmaM2} 
  \begin{split}
    \Sigma_\alpha'' = -\sdfrac{\#\cV\cdot 2p^2(p+1)}{D^4}
    \Big(\nnorm{\balpha}_E^2\Tr(\alpha)^2
    &+ N^2(p-1)\nnorm{\balpha}_E^2\\
    &+ N^2 (p+3)\Tr(\alpha)^2
    + N^4 (p-1)^2
    \Big).      
  \end{split}
\end{equation}

 \subsubsection{\texorpdfstring{The estimation of $\Sigma_\alpha'''$}{The estimation of Sigma_M'''}}
The sum in $\Sigma_\alpha'''$ is
\begin{align*}
  \sum_{x \in\cV}\Tr(\alpha-x)^4
  = \sum_{x \in\cV} \sum_{j=1}^{p-1} \sum_{k=1}^{p-1} 
    \sum_{m=1}^{p-1}\sum_{n=1}^{p-1}(a_j-x_j)(a_k-x_k)(a_m-x_m)(a_n-x_n).
\end{align*}
Expanding the product we find the following monomials:
\begin{equation}\label{eqSigma3MonA}
\begin{aligned}
  &-a_ja_ka_mx_n, &\quad& -a_ja_kx_ma_n, &\quad& -a_jx_ka_ma_n, &\quad& -x_ja_ka_ma_n,\\
  &-x_jx_kx_ma_n, &\quad& -x_jx_ka_mx_n, &\quad& -x_ja_kx_mx_n, &\quad& -a_jx_kx_mx_n,
\end{aligned}
\end{equation}
and
\begin{equation}\label{eqSigma3MonB}
\begin{aligned}
  &a_ja_kx_mx_n, &\qquad& a_jx_ka_mx_n, &\qquad& a_jx_kx_ma_n, &\qquad& a_ja_ka_ma_n,\\
  &x_ja_ka_mx_n, &\quad& x_ja_kx_ma_n, &\quad& x_jx_ka_ma_n, &\quad& x_jx_kx_mx_n.
\end{aligned}
\end{equation}
Because each monomial in~\eqref{eqSigma3MonA} has an odd number
of components of $\bx$, their total sum over $x,j,k,m,n$ cancels out
as seem in~\eqref{eqRemark1} and~\eqref{eqRemark3}.
Next, the sum  over $x,j,k,m,n$ of each of the first three monomials
on both rows of~\eqref{eqSigma3MonB} equals 
$\#\cV\cdot\Tr(\alpha)^2N^2(p-1)$, according 
to~\eqref{eqRemark22}.
Then, noting that $a_ja_ka_ma_n$ is constant and the sum of $x_jx_kx_mx_n$
is calculated in~\eqref{eqRemark4}, we have:
\begin{align*}
  \sum_{x \in\cV}\Tr(\alpha-x)^4
  = \#\cV\cdot\Tr(\alpha)^2N^2(p-1)\cdot 6
  + \#\cV\cdot\Tr(\alpha)^4 
   + \#\cV\cdot N^4(p-1)(3p-5).
\end{align*}
Therefore we obtain:
\begin{align}\label{eqSigmaM3}                     
    \Sigma_\alpha''' = \sdfrac{\#\cV (p+1)^2}{D^4}
    \Big(6\Tr(\alpha)^2N^2(p-1)
  + \Tr(\alpha)^4 
   + N^4(p-1)(3p-5)
    \Big).
\end{align}

\subsubsection{The conclusion of the proof of Lemma~\ref{Lemma2Moment}}

On inserting formulas~\eqref{eqSigmaM1}, \eqref{eqSigmaM2} and~\eqref{eqSigmaM3}
into~\eqref{eqMS123}, and then the results into~\eqref{eqMd4}, we have 
\begin{equation*}
  \begin{split}
 \sdfrac{1}{\#\cV}\sum_{x \in\cV} \distance_{p,N}^4(\alpha,x) 
 & =
\sdfrac{p^4}{D^4}
    \Big(\nnorm{\balpha}_E^4
    + 2N^2(p+1)\nnorm{\balpha}_E^2
  + N^4 (p-1)^2\Big)\\
&\quad
- \sdfrac{2p^2(p+1)}{D^4}
    \Big(\nnorm{\balpha}_E^2\Tr(\alpha)^2
    + N^2(p-1)\nnorm{\balpha}_E^2\\
&\phantom{\quad
- \sdfrac{2p^2(p+1)}{D^4}
    \Big(\nnorm{\balpha}_E^2\Tr(\alpha)^2+
}    
    + N^2 (p+3)\Tr(\alpha)^2
    + N^4 (p-1)^2
    \Big)\\
&\quad
+\sdfrac{(p+1)^2}{D^4}
    \Big(6\Tr(\alpha)^2N^2(p-1)
  + \Tr(\alpha)^4 
   + N^4(p-1)(3p-5)
    \Big).	
  \end{split}
\end{equation*}
On using this together with formula~\eqref{eqAverageLA2}
in~\eqref{eq2M1} we obtain~~\eqref{eqAverageL2M}, which concludes the proof of Lemma~\ref{Lemma2Moment}.
\end{proof}

\section{Distances from vertices to vertices}\label{SectionAMVV}
\subsection{The average distance between two vertices}
Let $A_{p,N}(\cV,\cV)$ denote the average normalized 
square distances between two vertices of $\scrB(p,N)$. 
The next lemma provides an exact formula for 
$A_{p,N}(\cV,\cV)$.
\begin{lemma}\label{LemmaAverageTwoVertices}
  Let $p$ be an odd prime, let $N\ge 1$ be integer
  and let $\cV$ be the set of the vertices of the origin-centered hypercube $\scrB(p,N)$.
Then
\begin{equation}\label{eqAverageVV}
 \begin{split}
 A_{p,N}(\cV,\cV)
   & := \frac{1}{\#\cV^2}
     \sum_{\alpha \in\cV}\sum_{\beta \in\cV}\distance_{p,N}^2 (\alpha,\beta)
     =  \sdfrac{1}{2}\big(1 - p^{-1} -p^{-2}\big)\,.   
 \end{split}
\end{equation}    
\end{lemma}
\begin{proof}
Using formula~\eqref{eqAverageLA}, we have
 \begin{equation}\label{eqAVV1}
  \begin{split}
 A_{p,N}(\cV,\cV)
   & = \frac{1}{\#\cV }
     \sum_{\alpha \in\cV} \frac{1}{\#\cV }
     \sum_{\beta \in\cV}
     \distance_{p,N}^2 (\alpha,\beta)\\
   & = \frac{1}{\#\cV }  \sum_{\alpha \in\cV}
   \Big(\distance_{p,N}^2 (0,\alpha) 
   + \sdfrac{1}{4}
   - \sdfrac{1}{4}p^{-1}
   - \sdfrac{1}{4}p^{-2}\Big)\\
   & = \frac{1}{\#\cV }  \sum_{\alpha \in\cV}
   \distance_{p,N}^2 (0,\alpha) 
   +  \sdfrac{1}{4}\big(1 - p^{-1} -p^{-2}\big)\,.
  \end{split}
\end{equation} 
With the diameter $D$ given by $D^2 = 4N^2p^2(p-1)$
and formula for the norm~\eqref{eqNorm_Formula},
the first term above is
\begin{align*}
   \frac{1}{\#\cV }  \sum_{\alpha \in\cV}
   \distance_{p,N}(0,\alpha)^2  
   &=
  \frac{1}{\#\cV D^2}  \sum_{\alpha \in\cV}
   \nnorm{\alpha}^2  \\
    &=
   \frac{1}{\#\cV D^2}  \sum_{\alpha \in\cV} 
   \big(p^2\nnorm{\balpha}_E^2-(p+1)\Tr(\alpha)^2\big)\\
   &=
   \frac{1}{\#\cV D^2}  \sum_{\alpha \in\cV} 
   p^2\sum_{j=1}^{p-1}\alpha_j^2
   - \frac{1}{\#\cV D^2}  \sum_{\alpha \in\cV} 
   (p+1)\sum_{j=1}^{p-1}\sum_{k=1}^{p-1}\alpha_j\alpha_k.
\end{align*}
Changing the order of summation and observing as in~\eqref{eqAS22}
and in the remark that followed that the sum of the off-diagonal
in the second term cancels out, we further obtain
\begin{align*}
   \frac{1}{\#\cV }  \sum_{\alpha \in\cV}
   \distance_{p,N}(0,\alpha)^2  
   &=
   \frac{1}{\#\cV D^2}
   p^2(p-1) \Big(2^{p-2}N^2+2^{p-2}(-N)^2\Big)\\
   &\quad- \frac{1}{\#\cV D^2}  
   (p+1)(p-1) \Big(2^{p-2}N^2+2^{p-2}(-N)^2\Big)\\
   &=\sdfrac{1}{4}\big(1 - p^{-1} -p^{-2}\big)\,.
\end{align*}
The lemma then follows on inserting this result into 
the right hand-side of~\eqref{eqAVV1}.
\end{proof}

\begin{lemma}\label{LemmaMomentLVV}
  Let $p$ be an odd prime, let $N\ge 1$ be integer
  and let $\cV$ be the set of the vertices of the origin-centered hypercube $\scrB(p,N)$.
Then
\begin{equation}\label{eqMomentLVV}
 \begin{split}
 L_{p,N}(\cV,\cV)
   & := \frac{1}{\#\cV^2}
     \sum_{\alpha \in\cV}\sum_{\beta \in\cV}\distance_{p,N}^4(\alpha,\beta)\\
    & =  \sdfrac{1}{4(p-1)}
     \big(p-2+ p^{-1} +2p^{-2}-5p^{-3}-4p^{-4}\big)\,.
 \end{split}
\end{equation}    
\end{lemma}
\begin{proof}
With the diameter $D$ satisfying $D^2 = 4N^2p^2(p-1)$
and using formula~\eqref{eqNorm_Formula} for the square of the norm,
the moment $L_{p,N}(\cV,\cV)$ can be rewritten as
\begin{equation}\label{eqM4A}
    \begin{split}
        L_{p,N}(\cV,\cV) &= \frac{1}{\#\cV^2D^4}
        \sum_{\alpha\in \cV}\sum_{\beta\in \cV}d(\alpha,\beta)^4\\
        &:= \frac{1}{\#\cV^2D^4}\, \big(\Sigma_\cV'+\Sigma_\cV''+\Sigma_\cV'''\big),
    \end{split}
\end{equation}
where
\begin{equation*}
  \begin{split}
    \Sigma_\cV' 
    & =  p^4\sum_{\alpha\in \cV}\sum_{\beta\in \cV}
        \nnorm{\bbeta-\balpha}_E^4,\\
        \Sigma_\cV'' 
    & = -2p^2(p+1)\sum_{\alpha\in \cV}\sum_{\beta\in \cV}
        \nnorm{\bbeta-\balpha}_E^2\Tr(\beta-\alpha)^2,\\
        \Sigma_\cV''' 
    & =  (p+1)^2\sum_{\alpha\in \cV}\sum_{\beta\in \cV}
        \Tr(\beta-\alpha)^4.
  \end{split}
\end{equation*}

In the following, let $\alpha=a_1\omega+\cdots+a_{p-1}\omega^{p-1}$
and 
$\beta = b_1\omega+\cdots+b_{p-1}\omega^{p-1}$ be the generic expression of
vertices written in the $\omega$-basis $\{\omega,\dots,\omega^{p-1}\}$
with coefficients $a_j,b_j\in\{-N,N\}$ for $1\le j\le p-1$.

The proof follows similar steps to those in the proof of Lemma~\ref{Lemma2Moment},
so we will highlight only where modifications arise because both parameters $\alpha$ and $\beta$ now run over $\cV$.
Here, also, the summations of different monomials are independent of each other and cancel out if any of their factors are raised to an odd power.
Also, as needed, we will distinguish terms that are on the diagonal from those that are not and sum them separately.
\subsubsection{\texorpdfstring{The estimation of $\Sigma_\cV'$}{The estimation of Sigma4'}}
The inner terms that are added in~$\Sigma_{\cV}'$~are:
\begin{equation*}
  \begin{split}
    \nnorm{\bbeta-\balpha}_E^4 
    & = 
    \big((b_1-a_1)^2+\cdots+(b_{p-1}-a_{p-1})^2\big)^2\\
    &= \sum_{j=1}^{p-1}\sum_{k=1}^{p-1}
      \big(a_j^2a_k^2+ a_j^2b_k^2 + a_k^2b_j^2 + b_j^2b_k^2\big)\\
      &\quad -2\sum_{j=1}^{p-1}\sum_{k=1}^{p-1}
      \big(a_ja_k^2b_j+ a_jb_jb_k^2 + a_j^2a_kb_k + a_kb_j^2b_k
      -2a_ja_kb_jb_k\big)\,.
  \end{split}
\end{equation*}
Expanding the polynomial and keeping only the terms that have a non-zero contribution, we see that $\Sigma_\cV'$ is
\begin{equation*}
  \begin{split}
   \Sigma_\cV'&= p^4 \sum_{j=1}^{p-1}\sum_{k=1}^{p-1}
       2^{2p-4}\bigg(
       \sum_{a_j\in \{-N, N\}} \sum_{a_k\in \{-N, N\}} a_j^2a_k^2
       + \sum_{a_j\in \{-N, N\}} \sum_{b_k\in \{-N, N\}} a_j^2b_k^2\\ 
       &\qquad\qquad\qquad\qquad\qquad\qquad
       + \sum_{a_k\in \{-N, N\}} \sum_{b_j\in \{-N, N\}} a_k^2b_j^2
       + \sum_{b_j\in \{-N, N\}} \sum_{b_k\in \{-N, N\}}b_j^2b_k^2\bigg)\\
       &\quad
       + 4p^4 \sum_{j=1}^{p-1}\sum_{\substack{k=1\\k=j}}^{p-1}
       2^{2p-4}
       \sum_{a_j\in \{-N, N\}} \sum_{a_k\in \{-N, N\}}
       \sum_{b_j\in \{-N, N\}} \sum_{b_k\in \{-N, N\}}
       a_ja_kb_jb_k\,.      
  \end{split}
\end{equation*}
Adding the sums together, we arrive at:
\begin{equation}\label{eqS4prim}
  \begin{split}
   \Sigma_\cV'& =  p^4 2^{2p-4} (p-1)^2\cdot (2N^2)^2 \cdot 4
                + 4p^4 2^{2p-4} (p-1) \cdot (2N^2)^2\\
         &=  p^5  (p-1)2^{2p}N^4.
  \end{split}
\end{equation}
 \subsubsection{\texorpdfstring{The estimation of $\Sigma_\cV''$}{The estimation of Sigma4''}}

The terms added in $\Sigma_\cV''$ are
\begin{equation*}
   \begin{split}
    \nnorm{\bbeta-\balpha}_E^2\Tr(\beta-\alpha)^2 = 
    \sum_{j=1}^{p-1}\sum_{m=1}^{p-1}\sum_{n=1}^{p-1}
    (a_j-b_j)^2(a_m-b_m)(a_n-b_n).
   \end{split}   
\end{equation*}
Expanding the product and keeping only the terms with even power factors that have a non-zero contribution, it follows:
\begin{equation*}
  \begin{split}
  \Sigma_\cV''
  & = -2p^2(p+1)2^{2p-3}
  \sum_{j=1}^{p-1}\sum_{a_j\in \{-N, N\}} a_j^4\cdot 2\\
  &\quad -2p^2(p+1)2^{2p-4} 
  \sum_{j=1}^{p-1}\sum_{\substack{m=1\\ m\neq j}}^{p-1} 
  \sum_{a_j\in \{-N, N\}}\sum_{a_m\in \{-N, N\}}a_j^2a_m^2\cdot 2\\
  &\quad -2p^2(p+1)2^{2p-4} 
    \sum_{j=1}^{p-1}\sum_{m=1}^{p-1} 
 \sum_{a_j\in \{-N, N\}}
 \sum_{b_m\in \{-N, N\}}a_j^2b_m^2\cdot 2\\
  &\quad -2p^2(p+1)2^{2p-4}   
  \sum_{j=1}^{p-1}   
  \sum_{a_j\in \{-N, N\}}\sum_{b_m\in \{-N, N\}}
  2   a_j^2b_m^2\cdot 2\,,
  \end{split}
\end{equation*}
where the multiplicative factors at the end of the lines
take into account the equal sums corresponding to 
the remaining symmetric terms.
This implies
\begin{equation}\label{eqSigma4secund}
  \begin{split}
  \Sigma_\cV''
  & = -4p^2(p+1)2^{2p-3}(p-1) \cdot 2N^4
   -4p^2(p+1)2^{2p-4}(p-1)(p-2)\cdot 4N^4\\    
  &\quad
  -4p^2(p+1)2^{2p-4}(p-1)^2\cdot 4N^4
 -8p^2(p+1)2^{2p-4}(p-1)\cdot 4N^4\\
 &= -p^3 (p-1)(p+1)2^{2p+1}N^4.     
  \end{split}
\end{equation}
 \subsubsection{\texorpdfstring{The estimation of $\Sigma_\cV'''$}{The estimation of Sigma4'''}}
The terms added in $\Sigma_\cV'''$ are:
\begin{equation*}
   \begin{split}
    \Tr(\beta-\alpha)^4 = 
    \sum_{j=1}^{p-1}\sum_{k=1}^{p-1}\sum_{m=1}^{p-1}\sum_{n=1}^{p-1}
    (a_j-b_j)(a_k-b_k)(a_m-b_m)(a_n-b_n).
   \end{split}   
\end{equation*}
Expanding the product and keeping in the sums only the term with even power factors, 
whose totals do not completely reduce, we obtain
\begin{equation*}
  \begin{split}
        \Sigma_\cV''' 
    & =  (p+1)^22^{2p-3}\sum_{j=1}^{p-1}
     \sum_{a_j\in \{-N, N\}} 
       a_j^4 \cdot 2 \\
    &\quad +  (p+1)^22^{2p-4}
    \sum_{j=1}^{p-1}\sum_{a_j\in \{-N, N\}}a_j^2 
     \sum_{\substack{m=1\\m\neq j}}^{p-1}\sum_{a_m\in \{-N, N\}}a_m^2 \cdot 3 \cdot 2 \\
    &\quad +
    (p+1)^22^{2p-4}\sum_{j=1}^{p-1} \sum_{m=1}^{p-1} 
    \sum_{a_j\in \{-N, N\}}a_j^2 \sum_{b_m\in \{-N, N\}}  b_m^2\cdot 6.
  \end{split}
\end{equation*}
Next this reduces to
\begin{equation}\label{eqSigma4tert}
  \begin{split}
        \Sigma_\cV''' 
    & =  2(p+1)^22^{2p-3}(p-1)\cdot 2N^4\\
    &\quad +  6(p+1)^22^{2p-4}(p-1)(p-2)\cdot 4N^4\\
    &\quad +
    6(p+1)^22^{2p-4}(p-1)^2 \cdot 4N^4\\
    &= (p-1)(p+1)^2(3p-4)2^{2p}N^4. 
  \end{split}
\end{equation}

 \subsubsection{\texorpdfstring{Completion of the proof}
 {Completion of the proof}}

Bringing together the conclusions from relations~\eqref{eqS4prim},
\eqref{eqSigma4secund}, and~\eqref{eqSigma4tert}
yields
\begin{equation*}
  \begin{split}
    \Sigma_\cV'+\Sigma_\cV''+ \Sigma_\cV''' 
    &= p^5  (p-1)2^{2p}N^4\\
    &\quad-p^3 (p-1)(p+1)2^{2p+1}N^4\\
    &\quad +(p-1)(p+1)^2(3p-4)2^{2p}N^4\\
    &= (p-1)\big(p^5-2p^4+p^3+2p^2-5p-4\big) 2^{2p}N^4.
  \end{split}
\end{equation*}
On inserting this result in~\eqref{eqM4A}, leads to~\eqref{eqMomentLVV}, thereby completing the proof of Lemma~\ref{LemmaMomentLVV}.
\end{proof}

\begin{lemma}\label{LemmaMomentMVV}
  Let $p$ be an odd prime, let $N\ge 1$ be integer
  and let $\cV$ be the set of the vertices of the origin-centered hypercube $\scrB(p,N)$.
Then
\begin{equation}\label{eqMomentMVV}
 \begin{split}
 M_{p,N}(\cV,\cV)
   & := \frac{1}{\#\cV^2}
     \sum_{\alpha \in\cV}\sum_{\beta \in\cV}
     \big(\distance_{p,N}^2(\alpha,\beta)-A_{p,N}(\cV,\cV)\big)^2\\
    & =  \sdfrac{1}{4(p-1)}
     \big(1 - p^{-2} -4p^{-3}-3p^{-4}\big)\,.
 \end{split}
\end{equation}    
\end{lemma}
\begin{proof}
By the definitions~\eqref{eqMomentMVV},~\eqref{eqMomentLVV}, and~\eqref{eqAverageVV}, we have:
\begin{align*}
   M_{p,N}(\cV,\cV)
   & := \frac{1}{\#\cV^2}
     \sum_{\alpha \in\cV}\sum_{\beta \in\cV}
     \distance_{p,N}^4(\alpha,\beta)
     -2 A_{p,N}(\cV,\cV)
     \frac{1}{\#\cV^2}
     \sum_{\alpha \in\cV}\sum_{\beta \in\cV}
     \distance_{p,N}^2(\alpha,\beta)
     + \big( A_{p,N}(\cV,\cV) \big)^2\\
  &\ = L_{p,N}(\cV,\cV) - \big( A_{p,N}(\cV,\cV) \big)^2.
\end{align*}
Then, formula~\eqref{eqMomentMVV} follows by inserting here the exact expressions given in~\eqref{eqMomentLVV} and~\eqref{eqAverageVV}.
This completes the proof of Lemma~\ref{LemmaMomentMVV}.
\end{proof}

\section{Proofs of the theorems}\label{SectionProofsOfTheorems}
The key result we aim to prove is that
nearly all the points in 
$\cV(p,N)$ are almost equally spaced as $p$ becomes sufficiently large.
Then, following similar steps, the same result can be shown to hold in $\scrB(p,N)$.
Next, as a consequence it will follow that almost all
triangles 
with vertices in $\cV(p,N)$ are almost equilateral.

To gain insight into the phenomenon, let us examine a closer more intuitive context.
As we called \textit{super-regular} a $K$-polytope 
with all its sides and diagonals almost equal,
in the case $K=2$, it is elementary that this happens with equality, since
all $2$-polytopes $AB$ meet the requirement 
\mbox{$[AB]=[BA]$} in most spaces.

\begin{figure}[htb]
 \centering
 \hfill
    \includegraphics[width=0.48\textwidth]{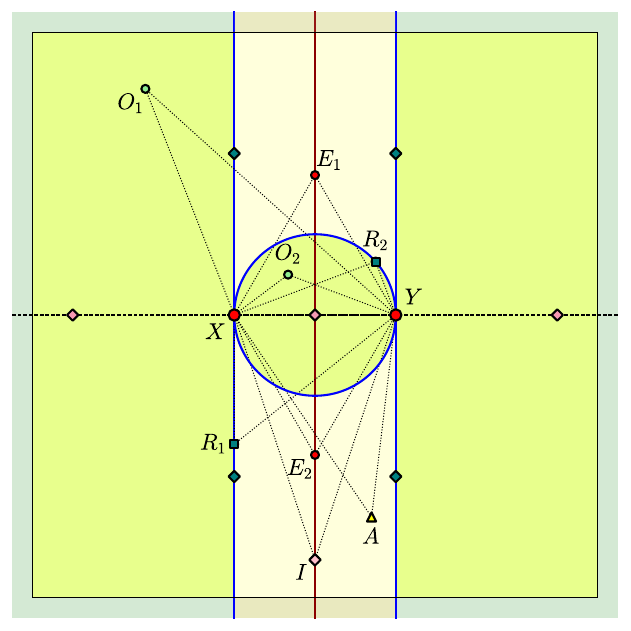}
\hfill\mbox{}
\caption{Distinguished types of triangles $\triangle (*XY)$, where
points $X$ and $Y$ are placed at 
a distance~$D$ from each other in a central symmetric position of 
the square 
$\Box{\smash[t]{\strut}}_{L}$
of side $L>D$, and $*$ stands for any point in 
$\Box{\smash[t]{\strut}}_{L}$.
} 
\label{FigurePointPositions}
 \end{figure}
If $K=3$, in the real plane and Euclidean distance, there are many equilateral triangles,
but their number pales in comparison to the multitude of other
shapes.
In Figure~\ref{FigurePointPositions},
two points $X$ and $Y$ situated at a distance $D$ apart are fixed 
in a square with side length $L$, 
while a third point $P$ has a random position within the square.
Then, the probability of 
$\triangle(PXY)$ being obtuse is $1-DL^{-1}+\pi D^2L^{-2}/4$, 
the probability of $\triangle(PXY)$ being acute is
$DL^{-1}-\pi D^2L^{-2}/4$, and, for all other types, 
the probability of a triangle being degenerate, rectangular, or isosceles is~$0$, 
although $P$ has infinitely many positions that could satisfy them.
Meanwhile, equilateral triangles are just two.
Then, even when allowing the sides to be almost equal in a certain limit 
$\epsilon$-$\delta$ process, the probability of randomly choosing three points 
to form a super-regular triangle is zero. 
This problem of counting or measuring the size of the set of
particular distinct geometric shapes has been explored as seen in a variety of works
by 
Hall~\cite{Hal1982}, 
Buchta~\cite{Buc1986}, 
Guy~\cite{Guy1993} 
Eisenberg and Sullivan~\cite{ES1996} 
Dunbar~\cite{Dun1997}, 
Mathai, Moschopoulos and Pederzoli~\cite{MMP1999}, 
Li and Qiu~\cite{LQ2016}, 
Burgstaller and Pillichshammer~\cite{BP2009}, 
Bat-Ochir~\cite{Bat2019}, 
B\"asel~\cite{Bas2021}. 

If $K=4$, it is impossible to find super-regular quadrilaterals 
in $\RR^2$, except for those that have all vertices nearly coincident. 
However, moving into three dimensions, we do find non-trivial quadrilaterals 
that have all six sides nearly equal, namely those that are close to regular tetrahedra. 
In higher dimensions, the surprising fact is that super-regular 
$K$-polytopes not only exist, but those that are not super-regular 
are in minuscule minority (see~\cite{ACZ2023,ACZ2024}),
and next we will see that the same holds in cyclotomic fields.

\subsection{Common distance from a point in the box to the vertices}

Let the integer $N\ge 1$ and the odd prime $p$ be fixed.
For a parameter $U>0$, the exact value of which we will determine later, we consider the set $\cU=\cU_{p,N}(U)\subset\cV(p,N)$ defined by
\begin{align*}
    \cU:= \Big\{
    x\in\cV(p,N) : \Big\vert\distance_{p,N}(\alpha,x)-\sqrt{A_{p,N}(\alpha,\cV)}\Big\vert\ge \sdfrac{1}{U}
    \Big\}\,.
\end{align*}

For every integer $N\ge 1$, every prime $p$
and every $\alpha\in\scrB(p,N)$,
by Lemma~\ref{LemmaAverage} we find that the 
average normalized distance from~$\alpha$ to $\cV=\cV(p,N)$ is bounded by
\begin{align}\label{eqBoundAverage}
   \sdfrac{5}{36}
   \le \sdfrac{1}{4}\big(1-p^{-1}-p^{-2}\big)
   \le A_{p,N}(\alpha,\cV) 
   \le 1.
\end{align}
Then, following~\eqref{eqMbound} we have:
\begin{equation}\label{eqMlb1}
  \begin{split}
 \sdfrac{3}{p}
 &\ge \frac{1}{\#\cV}
     \sum_{x \in\cV}
     \big(\distance_{p,N}^2(\alpha,x)-A_{p,N}(\alpha,\cV)\big)^2
 \ge \frac{1}{\#\cV}
      \sum_{x \in\cV\bigcap \cU}
     \big(\distance_{p,N}^2(\alpha,x)-A_{p,N}(\alpha,\cV)\big)^2.
  \end{split}
\end{equation}
Since
\begin{align*}
  \sqrt{A_{p,N}(\alpha,\cV)}
    &\Big\vert\distance_{p,N}(\alpha,x)-\sqrt{A_{p,N}(\alpha,\cV)}\Big\vert\\
    &\le
  \Big(\distance_{p,N}(\alpha,x)+\sqrt{A_{p,N}(\alpha,\cV)}\Big)
  \cdot
  \Big\vert\distance_{p,N}(\alpha,x)-\sqrt{A_{p,N}(\alpha,\cV)}\Big\vert\\
  & =
  \Big\vert\distance^2_{p,N}(\alpha,x)-A_{p,N}(\alpha,\cV)\Big\vert,
\end{align*}
the right hand-side of the estimate~\eqref{eqMlb1} 
can be further lowered, so that we obtain
\begin{equation}\label{eqMlb2}
  \begin{split}
 \sdfrac{3}{p}
 &\ge \frac{1}{\#\cV}
      \sum_{x \in \cU}
      A_{p,N}(\alpha,\cV)
     \Big(\distance_{p,N}(\alpha,x)-
     \sqrt{A_{p,N}(\alpha,\cV)}\Big)^2
 \ge \frac{1}{\#\cV}
 \cdot\#\cU\cdot A_{p,N}(\alpha,\cV) \cdot \sdfrac{1}{U^2}.
  \end{split}
\end{equation}
On combining~\eqref{eqBoundAverage} and~\eqref{eqMlb2}
we find that
\begin{align*}
   \frac{1}{\#\cV(p,N)} \cdot\#\cU
   <  22\cdot \sdfrac{U^2}{p}.
\end{align*}
Letting now $U=p^{\eta}$ for some $\eta$ that satisfies
$0<\eta< 1/2$, we obtain the following result.
\begin{theorem}\label{Theorem4}
Let $A_{p,N}(\alpha,\cV)$ be the average of the squares of
the normalized distances from~$\alpha$ to 
$\cV=\cV(p,N)$, which is
the set of vertices of the cyclotomic box $\scrB(p,N)$.
Let $\eta\in (0,1/2)$ be fixed.
Then for all integers $N\ge 1$, for 
all odd primes $p$, and for all $\alpha\in\scrB(p,N)$, we have 
\begin{equation*}\label{eqPTV1}
  \sdfrac{1}{\#\cV(p,N)} 
  \#\Big\{
    x\in\cV(p,N) : 
    \distance_{p,N}(\alpha,x) \in
    \Big[\sqrt{A_{p,N}(\alpha,\cV)}-\sdfrac{1}{p^{\eta}},
    \sqrt{A_{p,N}(\alpha,\cV)}+\sdfrac{1}{p^{\eta}}
        \Big]
    \Big\}
    \ge 1- \sdfrac{22}{p^{1-2\eta}}\,.
\end{equation*}
\end{theorem}

\subsection{Isosceles triangles with bases in the set of vertices}\label{SectionIsosceles}
For any $\alpha\in\scrB(p,N)$, we let
$\cT=\cT_{p,N}(\alpha,\cV)$ denote the set of triangles 
$\triangle = (\alpha, \beta_1,\beta_2)$
with vertices $\alpha$ and $\beta_1,\beta_2\in\cV = \cV(p,N)$.
Then
\begin{align*}
 &\sdfrac{1}{\#\cV^2} 
 \#\Big\{\triangle\in\cT
     : \distance_{p,N}(\alpha,\beta_1),
     \distance_{p,N}(\alpha,\beta_2)
     \in \Big[\sqrt{A_{p,N}(\alpha,\cV)}-\sdfrac{1}{p^\eta},
     \sqrt{A_{p,N}(\alpha,\cV)}+\sdfrac{1}{p^\eta}\Big]
    \Big\}\\
    \ge & 
  \sdfrac{1}{\#\cV}
  \#\Big\{\beta_1\in\cV
     : \distance_{p,N}(\alpha,\beta_1)
     \in \Big[\sqrt{A_{p,N}(\alpha,\cV)}-\sdfrac{1}{p^\eta},
     \sqrt{A_{p,N}(\alpha,\cV)}+\sdfrac{1}{p^\eta}\Big]
    \Big\}\\
    &\times
    \sdfrac{1}{\#\cV}
  \#\Big\{\beta_2\in\cV
     : \distance_{p,N}(\alpha,\beta_2)
     \in \Big[\sqrt{A_{p,N}(\alpha,\cV)}-\sdfrac{1}{p^\eta},
     \sqrt{A_{p,N}(\alpha,\cV)}+\sdfrac{1}{p^\eta}\Big]
    \Big\}\,.
\end{align*}
Employing the estimate in Theorem~\ref{Theorem4}, 
we further find that the above is larger than:
\begin{align*}
    \Big( 1- \sdfrac{22}{p^{1-2\eta}}\Big)^2
    \ge  1- \sdfrac{44}{p^{1-2\eta}}.  
\end{align*}
This shows that as $p$ becomes large, almost all triangles $\triangle = (\alpha, \beta_1,\beta_2)$
with  $\beta_1$ and $\beta_2$ vertices of $\scrB(p,N)$ are almost isosceles in shape.
\begin{theorem}\label{TheoremIsoscelesTriangles}
Let $\alpha\in\scrB(p,N)$ and denote by
$\cT_{p,N}(\alpha,\cV)$ the set of triangles 
$\triangle = (\alpha, \beta_1,\beta_2)$
with two vertices $\beta_1,\beta_2\in\cV(p,N)$,
the set of vertices of $\scrB(p,N)$,
and the third at  $\alpha$.
Let~$A_{p,N}(\alpha,\cV)$ be the average of the squares of the normalized distances from~$\alpha$ to~$\cV(p,N)$.
Let $\eta\in(0,1/2)$ be fixed. 
Let $\cI_{p,N}(\eta)$ denote the neighborhood interval of
$\sqrt{A_{p,N}(\alpha,\cV)}$:
\begin{align*}
   \cI_{p,N}(\eta) : = 
   \Big[\sqrt{A_{p,N}(\alpha,\cV)}-\sdfrac{1}{p^\eta},
     \sqrt{A_{p,N}(\alpha,\cV)}+\sdfrac{1}{p^\eta}\Big]\,.
\end{align*}
Then for any integer $N\ge1$, any odd prime $p$, 
and any $\alpha\in\scrB(p,N)$, we have
\begin{align*}
 &\sdfrac{1}{\#\cV(p,N)^2} 
 \#\big\{(\alpha, \beta_1,\beta_2)\in\cT_{p,N}(\alpha,\cV)
     : \distance_{p,N}(\alpha,\beta_1),
     \distance_{p,N}(\alpha,\beta_2)
     \in \cI_{p,N}(\eta)
    \big\}
   \ge  1- \sdfrac{44}{p^{1-2\eta}}\,.
\end{align*}
\end{theorem}

Note that in particular Theorem~\ref{TheoremIsoscelesTriangles}
shows that almost all lateral faces of almost all 
pyramids whose polytopal bases have vertices in $\cV(p,N)$
are almost isosceles triangles, 
thereby proving Theorem~\ref{TheoremIsosceles} and half of Corollary~\ref{CorollaryPyramids}.

\subsection{Triangles with a side that joins the center to a vertex.}
Although there are many similarities, there are also differences between the super-regularity properties in high dimensional hypercubes in Euclidean spaces~\cite{ACZ2023, ACZ2024} compared to those in cyclotomic fields.
One such difference is that in the Euclidean hypercube
$\cW=[0,N]^d\cap\NN^d$, the distance between the center
$\bc =\big(\frac{N}{2},\dots,\frac{N}{2}\big)$ 
and a vertex of $\cW$ is the same for any vertex, while
in the cyclotomic box $\scrB(p,N)$ the center, which is the origin $O$ in the cyclotomic case, 
is not equidistant from the vertices.
Indeed, if  $\beta=b_1\omega+\cdots+b_{p-1}\omega^{p-1}$
is a vertex of $\scrB(p,N)$, 
then $b_1,\dots,b_{p-1}\in\{-N,N\}$, so that
\begin{align*}
  \nnorm{\bbeta}^2_E=(p-1)N^2\ \ \text{ and }\ \ 
 0\le \Tr(\beta)^2\le (p-1)^2N^2.
\end{align*}
Then, since $\nnorm{\beta}^2=p^2\nnorm{\bbeta}^2_E-(p+1)\Tr(\beta)^2$, the distance $d(O,\beta)$ satisfies the following inequalities:
\begin{align}\label{eqDistanceOVInequalities}
(p-1)N^2\le 
d(O,\beta)^2\le (p-1) p^2N^2.
\end{align}
Additionally, the values of $d(O,\beta)^2$
are spread out over the entire range determined by the inequalities above, while the $(p-1)$-tuple of the coefficients of $\beta$
run over $\{-N,N\}^{p-1}$.
And yet this does not contradict Theorem~\ref{Theorem4}, because in fact almost all $(p-1)$-tuples are almost \textit{balanced}, having approximately the same number of components equal to both $-N$ and~$N$.
This is due to the fact that the binomial coefficients near the middle overwhelmingly dominate the sum of the others (see the proof and the application discussed in~\mbox{\cite[Theorem 5]{BCZ2024}}).
\begin{remark}\label{RemarkNeglijableTrace}
Let us note that this means that the inequality on the right side 
of~\eqref{eqDistanceOVInequalities} approaches equality for almost all vertices
and that in measuring distances the influence of the trace
becomes increasingly negligible as $p$ gets larger.
\end{remark}

We turn now to the problem of estimating 
the size of the center angle 
$\widehat{\alpha O \beta}$, where 
$\alpha\in\scrB(p,N)$ is fixed, and~\mbox{$\beta\in\cV(p,N)$}.
From Theorem~\ref{Theorem4}
we know that if $\alpha\in\scrB(p,N)$ is fixed, 
then for every $\eta\in (0,1/2)$
almost all vertices $\beta\in\cV(p,N)$ satisfy
\begin{align*}
  \distance_{p,N}^2(\alpha,\beta) 
    &= A_{p,N}(\alpha,\cV) +
    E'_{\alpha}\big(p^{-\eta}\big),
\end{align*}
where $\big|E'_{\alpha}(t)\big|\le 3t$ for $0<t<1$.
Then, employing formula~\eqref{eqAverageLA} for the
average yields:
\begin{equation}\label{eqdAab}
  \begin{split}
  \distance_{p,N}^2(\alpha,\beta) 
      &=  \distance_{p,N}^2(O,\alpha)
      +\sdfrac{1}{4}\Big(1-\sdfrac{1}{p} -\sdfrac{1}{p^2} \Big)
      + E'_{\alpha}\big(p^{-\eta}\big)\\
       &=  \distance_{p,N}^2(O,\alpha)+\sdfrac{1}{4}
      + E''_{\alpha}\big(p^{-\eta}\big),
  \end{split}
\end{equation}
where $\big|E''_{\alpha}(t)\big|\le 4t$ for $0<t<1$.
In particular, if $\alpha$ is the origin
and $\beta$ is a vertex, we have
\begin{equation}\label{eqdAOb}
  \begin{split}
  \distance_{p,N}^2(O,\beta)
       &= \sdfrac{1}{4}
      + E''_{O}\big(p^{-\eta}\big).
  \end{split}
\end{equation}
Then by substituting~\eqref{eqdAab} and~\eqref{eqdAOb} 
into the law of cosines, we obtain:
\begin{align*}
  |\cos \widehat{(\alpha O\beta)}|
  &= \sdfrac{1}{2}
  \big( \distance_{p,N}^2(O,\alpha) 
  + \distance_{p,N}^2(O,\beta)
  - \distance_{p,N}^2(\alpha,\beta)\big) 
  \distance_{p,N}^{-1}(O,\alpha)
  \distance_{p,N}^{-1}(O,\beta)
  \\
  &\le \sdfrac{1}{2}\big(
  \big|E''_{O}\big(p^{-\eta}\big| + 
    \big|E''_{\alpha}\big(p^{-\eta}\big|)
  \big) 
  \cdot \distance_{p,N}^{-1}(O,\alpha)
  \Big(\sdfrac{1}{4}+E''_{O}\big(p^{-\eta}\big)\Big)^{-1/2}\\
  &\le C_3 p^{-\eta}  \distance_{p,N}^{-1}(O,\alpha),
\end{align*}
where $C_3>0$ is an absolute constant.
Now if $\alpha\in\scrB(p,N)$ is not too close to $O$, so that it satisfies condition 
$\distance_{p,N}(O,\alpha)\ge p^{\gamma-\eta}$ for some $\gamma>0$, it follows
\begin{align*}
  |\cos \widehat{(\alpha O \beta)}|
  &\le C_3 p^{-\gamma}.
\end{align*}
However, note that $\gamma$ cannot be taken too large, 
because following~\eqref{eqDistanceOVInequalities} we know that
the normalized distance from from $O$ to $\alpha$ cannot 
be larger than $1/2$. Therefore $\gamma$ needs to satisfy condition $p^{\gamma-\eta}\le 1/2$.

The conclusion reached in this way, 
the precise formulation of which 
will follow, shows us in particular that almost all angles 
$\widehat{\alpha O \beta}$ with $\beta\in\cV(p,N)$ 
are almost right angles, as stated 
in Theorem~\ref{TheoremAngle}.
\begin{theorem}\label{TheoremRightAngles}
Let $\eta\in(0,1/2)$ be fixed. 
Suppose $p\ge 2^{1/\eta}$ is an odd prime and let 
\begin{align*}
   \gamma\in\Big(0,\,\eta-\sdfrac{\log 2}{\log p}\Big] \,.
\end{align*}
Additionally, for any integer $N\geq1$, 
let $\alpha\in\scrB(p,N)$ satisfying condition
$\distance_{p,N}(O,\alpha)\ge p^{\gamma-\eta}$
be fixed.
Then, there exist absolute constants $C_3>0$
and $C_4>0$
such that
\begin{equation*}
    \frac{1}{\#\cV(p,N)}\#\left\{ \beta\in\cV(p,N) : \big|\cos(\widehat{\alpha O\beta})\big| 
    \le C_3p^{-\gamma} \right\}
        \ge 1-\sdfrac{C_4}{p^{1-2\eta}}.
\end{equation*}
\end{theorem}

\subsection{Super-regularity of distances between vertices}

For an integer $N\ge 1$, an odd prime~$p$,
and a parameter $W>0$, whose exact value will be specified later, we consider the set $\cW=\cW_{p,N}(W)\subset\cV(p,N)^2$ defined by
\begin{align*}
    \cW:= \Big\{
    (\beta_1, \beta_2)\in\cV(p,N)^2 : \Big\vert\distance_{p,N}(\beta_1, \beta_2)-\sqrt{A_{p,N}(\cV,\cV)}\Big\vert\ge \sdfrac{1}{W}
    \Big\}\,.
\end{align*}

From Lemma~\ref{LemmaAverageTwoVertices} we know the following bounds
for the average of the normalized distance between any
two vertices of~$\scrB(p,N)$:
\begin{align}\label{eqBoundAverageVV}
   \sdfrac{5}{18}
   \le \sdfrac{1}{2}\big(1-p^{-1}-p^{-2}\big)
   = A_{p,N}(\cV,\cV) 
   \le 1.
\end{align}
Then, the exact formula~\eqref{eqMomentMVV} implies:
\begin{equation}\label{eqMlb1VV}
  \begin{split}
  \frac{ 1}{4(p-1)}
   \big(1-p^{-2}-4p^{-3}-3p^{-4}\big)
  &= \frac{1}{\#\cV^2}
     \sum_{(\beta_1,\beta_2)\in\cV^2}
     \big(\distance_{p,N}^2(\beta_1,\beta_2)-A_{p,N}(\cV,\cV)\big)^2\\
 &\ge \frac{1}{\#\cV^2}
      \sum_{(\beta_1,\beta_2) \in\cV^2\bigcap \cW}
     \big(\distance_{p,N}^2(\beta_1,\beta_2)-A_{p,N}(\alpha,\cV)\big)^2.
  \end{split}
\end{equation}
Note that
\begin{align*}
  \sqrt{A_{p,N}(\cV,\cV)}
    &\Big\vert\distance_{p,N}(\beta_1,\beta_2)-\sqrt{A_{p,N}(\cV,\cV)}\Big\vert\\
    &\le
  \Big(\distance_{p,N}(\beta_1,\beta_2)+\sqrt{A_{p,N}(\cV,\cV)}\Big)
  \cdot
  \Big\vert\distance_{p,N}(\beta_1,\beta_2)-\sqrt{A_{p,N}(\cV,\cV)}\Big\vert\\
  & =
  \Big\vert\distance^2_{p,N}(\beta_1,\beta_2)-A_{p,N}(\cV,\cV)\Big\vert.
\end{align*}
Then~\eqref{eqMlb1VV} implies
\begin{equation}\label{eqMlb2VV}
  \begin{split}
  \frac{ 1}{4(p-1)}
   \big(1-p^{-2}-4p^{-3}-3p^{-4}\big)
  &\ge \frac{1}{\#\cV^2}
     \sum_{(\beta_1,\beta_2)\in\cW}
     \big(\distance_{p,N}^2(\beta_1,\beta_2)-A_{p,N}(\cV,\cV)\big)^2\\
  & \ge \frac{1}{\#\cV^2}
 \cdot\#\cW\cdot A_{p,N}(\cV,\cV) \cdot \sdfrac{1}{W^2}.
  \end{split}
\end{equation}
Next, using the bound~\eqref{eqBoundAverageVV} in~\eqref{eqMlb2VV}
it follows
\begin{align*}
   \frac{1}{\#\cV(p,N)^2} \cdot\#\cW
   \le  \sdfrac{29}{30}\cdot \sdfrac{W^2}{p-1}
   \le \sdfrac{2}{p}W^2
\end{align*}
for all odd primes $p$.
In the next theorem we write the conclusion that follows
letting $W=p^{\eta}$ for some $\eta\in (0,1/2)$.
\begin{theorem}\label{Theorem5}
Let $A_{p,N}(\cV,\cV)$ be the average of the squares of
the normalized distances between any two vertices  
in $\cV=\cV(p,N)$.
Let $\eta\in (0,1/2)$ be fixed.
Then, for all integers~\mbox{$N\ge 1$} and for 
all odd primes $p$, we have 
\begin{equation*}\label{eqPTVV}
  \sdfrac{1}{\#\cV^2} 
  \#\Big\{
    (\beta_1,\beta_2)\in\cV^2 : 
    \distance_{p,N}(\beta_1,\beta_2) \in
    \Big[\sqrt{A_{p,N}(\cV,\cV)}-\sdfrac{1}{p^{\eta}},
    \sqrt{A_{p,N}(\cV,\cV)}+\sdfrac{1}{p^{\eta}}
        \Big]
    \Big\}
    \ge 1- \sdfrac{2}{p^{1-2\eta}}\,.
\end{equation*}
\end{theorem}

\subsection{Super-regular polytopes - proof of Theorem~\ref{TheoremMaineps}}\label{subsectionKpolytopes}

Following Lemma~\ref{LemmaAverageTwoVertices}, that gives
the expresion  
$A_{p,N}(\cV,\cV)= 1/2 -p^{-1}/2-p^{-2}/2$, 
we can also proceed as described in the proof of Theorem~\ref{Theorem5} to determine the count of the 
pairs where 
\begin{align*}
  \distance_{p,N}(\beta_1,\beta_2) \in
    \Big[\sdfrac{1}{\sqrt{2}}-\sdfrac{1}{T},
    \sdfrac{1}{\sqrt{2}}+\sdfrac{1}{T}
        \Big]\,.
\end{align*}
Subsequently, we can adapt by iteratively applying the arguments from Section~\ref{SectionIsosceles}
multiple times on pairs of 
$(\beta_1,\beta_2)\in\cV(p,N)^2$ belonging to linked groups 
that are vertices of polytopes.
By doing so, the only drawback is that the error term 
grows proportionally to the number of applications.
Noting that the number of edges of a $K$-polytope is 
$\binom{K}{2}$, it follows that if $K\ge 2$, almost all 
$K$-polytopes with vertices in $\cV(p,N)$ are super-regular as $p$ tends to infinity.
The precise result is stated as the following theorem.

\begin{theorem}\label{TheoremKVV}
There exist absolute positive constants $C_5$ and $C_6$ 
such that for any prime $p\ge 3$, any integer $N\ge1$, 
any integer $K\ge 2$, and
any real number $T$ such that 
\begin{equation}\label{eqTmin2}
  \begin{split}
   & 1<T\le  C_5p,
  \end{split}  
\end{equation}
we have
\begin{equation}\label{eqTheoremKVV}
  \begin{split}
     	\sdfrac{1}{\#\cV(p,N)^{K}}   
    \#\Big\{(\beta_1,\dots,\beta_K)\in\cV(p,N)^K : &
        \max_{1\le j<k\le K}\Big|\distance_{p,N}(\beta_j,\beta_k)-\sdfrac{1}{\sqrt{2}}\Big|
 \le \sdfrac{1}{T}\Big\}        \\
               &\ge  1-\sdfrac{C_6K(K-1)T^2}{p}\,. 
  \end{split}
\end{equation}
\end{theorem}
We remark that for a non-trivial result in~\eqref{eqTheoremKVV} we need to take $T\ll p^{1/2}/K$.
Finally, note that Corollary~\ref{CorollaryKeps}, 
as well as Theorem~\ref{TheoremMaineps},
follows from Theorem~\ref{TheoremKVV}
in the same way as  Theorems~\ref{TheoremIsosceles} and~\ref{TheoremAngle} follow
from Theorems~\ref{TheoremIsoscelesTriangles} and~\ref{TheoremRightAngles}, by rephrasing with $\varepsilon$ instead of $1/T$.
Let us also observe that by establishing the almost regularity of nearly all polygons with vertices in set $\cV(p,N)$, 
Theorem~\ref{TheoremKVV} also concludes the proof of Corollary~\ref{CorollaryPyramids} concerning the regularity of pyramids
$(\alpha,\beta_1, \dots, \beta_K)$ with $\alpha\in\scrB(p,N)$ and
$\beta_1, \dots, \beta_K\in\cV(p,N)$.

\vspace{11mm}



\vspace{26pt} 

\end{document}